\documentclass[11pt]{amsart}
\usepackage{amsmath,amssymb,amsthm,mathrsfs}
\usepackage[all]{xy}
\usepackage{mathtools}
\usepackage{xcolor}

\newtheorem*{Thm*}{Theorem}
\newtheorem*{prop*}{Proposition}
\newtheorem*{cor*}{Corollary}
\newtheorem{Thm}{Theorem}[section]
\newtheorem{Prop}[Thm]{Proposition}
\newtheorem{Cor}[Thm]{Corollary}
\newtheorem{Lem}[Thm]{Lemma}
\numberwithin{equation}{section}

\theoremstyle{definition}
\newtheorem{Conj}[Thm]{Conjecture}
\newtheorem{Def}[Thm]{Definition}
\newtheorem{Ex}[Thm]{Example}

\newtheorem{Rem}[Thm]{Remark}

\newcommand{\C}{\mathbb{C}}
\newcommand{\R}{\mathbb{R}}
\newcommand{\Q}{\mathbb{Q}}
\newcommand{\HH}{\mathbb{H}}
\newcommand{\Z}{\mathbb{Z}}
\newcommand{\N}{\mathbb{N}}
\newcommand{\NN}{\mathcal{N}}
\newcommand{\DD}{\mathcal{D}}

\newcommand{\PP}{\mathcal{P}}

\newcommand{\ZZ}{\mathcal{Z}}
\newcommand{\Stab}{\mathrm{Stab}}

\newcommand{\Hom}{\mathrm{Hom}}
\newcommand{\GL}{\mathrm{GL}}
\newcommand{\SL}{\mathrm{SL}}
\newcommand{\Aut}{\mathrm{Aut}}
\newcommand{\ch}{\mathrm{ch}}
\newcommand{\rank}{\mathrm{rank}}
\newcommand{\tr}{\mathrm{tr}}

\newcommand{\Td}{\mathrm{Td}}

\newcommand{\Cof}{\mathrm{Cof}}
\newcommand{\re}{\mathrm{Re}}
\newcommand{\im}{\mathrm{Im}}

\newcommand{\CY}{\mathrm{CY}}

\newcommand{\Attr}{\mathrm{Attr}}

\begin{document}
\title{Attractor mechanisms of moduli spaces of Calabi--Yau 3-folds}
\author{Yu-Wei Fan \ \ \ Atsushi Kanazawa}
\date{}

\maketitle

\begin{abstract}
We investigate the complex and K\"ahler attractor mechanisms of moduli spaces of Calabi--Yau 3-folds. 
The complex attractor mechanism was previously studied by Ferrara--Kallosh--Strominger, Moore and others in string theory. 
It is concerned with the minimizing problems of the normalized central charges of $3$-cycles and defines a new interesting class of Calabi--Yau 3-folds called, the complex attractor varieties. 
In light of mirror symmetry, we introduce the K\"ahler attractor mechanism and define the K\"ahler attractor varieties. 
The complex and  K\"ahler attractor varieties are expected to possess very rich structures, in particular certain complex and K\"ahler rigidities.
\end{abstract}


\section{Introduction}
Let $X$ be a projective Calabi--Yau 3-fold. 
Let $\pi: \widetilde{\mathfrak{M}}_{\mathrm{Cpx}} \rightarrow \mathfrak{M}_{\mathrm{Cpx}}$ be the universal covering of the complex moduli space $\mathfrak{M}_{\mathrm{Cpx}}$ of $X$. 
The normalized central charge of a $3$-cycle $\gamma \in H_3(X,\Z)$ is defined by 
$$
Z(\Omega_{X_z},\gamma)= e^{\frac{K^B(z)}{2}}\int_\gamma \Omega_{X_z}
$$
where $\Omega_{X_z}$ is a holomorphic volume form of $X_z$ and $K^B(z)$ is the Weil--Petersson potential on $\widetilde{\mathfrak{M}}_{\mathrm{Cpx}}$.   
It induces a function $|Z(-,\gamma)|:\widetilde{\mathfrak{M}}_{\mathrm{Cpx}} \rightarrow \R_{\ge0}$, called the mass function of $\gamma$, 
and we are interested in its stationary points, called the attractors. 
Our investigation is motivated by the study of black holes in string theory (Ferrara--Kallosh--Strominger \cite{FKS}),   
where it is of great interest to find a $3$-cycle $\gamma \in H_3(X,\Z)$ which supports a BPS state.  
Finding stationary points of the mass function $|Z(-,\gamma)|$ is a purely mathematical problem and can be answered in parts by using the attractor mechanism investigated by Moore in his unpublished article \cite{Moo}. 
The Calabi--Yau 3-folds corresponding to the attractors are called the attractor varieties. 
This new class of Calabi--Yau 3-folds are conjectured to posses very rich structures. 

Moore's article \cite{Moo} is full of beautiful insights and he posed many interesting mathematical questions (the attractor conjectures) pertaining to the arithmetic nature of the attractor varieties. 
In fact, the attractor varieties can be considered as a vast generalization of the rigid Calabi--Yau 3-folds.

In light of of mirror symmetry, which is a duality between complex and K\"ahler (symplectic) geometries of distinct Calabi--Yau manifolds, 
a natural question is, what is the mirror of the attractor mechanism? 
In this article, we introduce the K\"ahler attractor mechanism of the K\"ahler moduli space 
and develop parallel theories to the complex side.   
Moreover, the K\"ahler attractor mechanism leads us to the idea of rigid K\"ahler structures, which should be mirror to the rigid complex structures. 
This direction of research is further carried out from the viewpoint of generalized Calabi--Yau geometry in a separate article \cite{Kan}. 

The present work is based on our previous work \cite{FKY}. 
It investigates the A-model Weil--Petersson geometry on the K\"ahler moduli space (or more precisely the space of Bridgeland stability conditions), 
which is supposed to be mirror to the classical Weil--Petersson geometry on the complex moduli spaces \cite{Tia, Tod}. 

The objective of this article is twofold. 
The first is to provide mathematical foundations of the complex attractor mechanism (Section \ref{Complex attractor mechanism}). 
The second is to introduce the K\"ahler attractor mechanism inspired by mirror symmetry (Section \ref{Kahler attractor}).  

\begin{table}[htb]
\begin{center}
  \begin{tabular}{c|c}
    (A-side) & (B-side) \\
    K\"ahler moduli $\mathfrak{M}_{\mathrm{Kah}}$ &  complex moduli $\mathfrak{M}_{\mathrm{Cpx}}$  \\ \hline
    Weil--Petersson metric $g^A$  & Weil--Petersson metric $g^B$  \\ 
    (Fan--Kanazawa--Yau \cite{FKY}) & (Tian \cite{Tia}, Todorov \cite{Tod}) \\ \hline
    K\"ahler attractor mechanism & complex attractor mechanism \\
    (Fan--Kanazawa  [present work]) & (Moore \cite{Moo})
  \end{tabular}
  \end{center}
\end{table}

 \subsection*{Structure of article}
Section \ref{Complex attractor mechanism} provides mathematical foundations of the complex attractor mechanism based on \cite{Moo}. 
Section \ref{WP stability conditions} is a brief review of our previous work \cite{FKY} on the Weil--Petersson geometry by means of the Bridgeland stability conditions. 
Section \ref{Kahler attractor} introduces the K\"ahler attractor mechanism and develops parallel theories to the complex side. 
We finally compare the complex and K\"ahler attractor mechanisms from the view point of mirror symmetry.

\subsection*{Notation and conventions}
Throughout the article, we work over complex numbers $\C$. 
A Calabi--Yau $n$-fold is an $n$-dimensional K\"ahler manifold whose canonical bundle is trivial. 
$\ch(-)$ denotes the Chern character and $\Td_X$ denotes the Todd class of $X$. 
For $R=\Z,\Q,\R$, $H^{i,i}(X,R)$ denotes the intersection $H^{i,i}(X)\cap H^{2i}(X,R)$. 
$\mathfrak{H}_g$ denotes the Siegel upper half-space of degree $g$.

\subsection*{Acknowledgement}
First and foremost, we would like to express our  gratitude to Shing-Tung Yau, who first drew our attention to the attractor mechanism. 
Our thanks also go to Shinobu Hosono and Hiroshi Iritani for very useful discussions. 
A.K. was supported in part by the JSPS Grant-in-Aid Wakate(B)17K17817 and Leading Initiative for Excellent Young Researchers Grant (Kyoto University).   


\section{Complex attractor mechanism} \label{Complex attractor mechanism}

\subsection{Foundation of complex attractor mechanism}

Let $X$ be a projective Calabi--Yau 3-fold and $ \mathfrak{M}_{\mathrm{Cpx}}$ the complex moduli space of $X$. 
We consider the vector bundle $\mathcal{H}=R^3\pi_*\underline{\C} \rightarrow  \mathfrak{M}_{\mathrm{Cpx}}$ 
equipped with a natural Hodge filtration $F^3\mathcal{H} \subset \dots \subset F^0\mathcal{H}$ of weight $3$. 
By the Calabi--Yau condition, the first piece of the filtration defines a holomorphic line bundle $\mathcal{L}=F^3\mathcal{H}\rightarrow  \mathfrak{M}_{\mathrm{Cpx}}$, which we call the Hodge bundle. 
It is classically known that $\mathfrak{M}_{\mathrm{Cpx}}$ carries a natural K\"ahler metric, called the Weil--Petersson metric $g^B$,  
whose K\"ahler potential is given by 
\begin{equation}
K^B(z)=- \log(\sqrt{-1}\int_{X_z} \Omega_{X_z} \wedge \overline{ \Omega_{X_z}}), \label{WP potential1} 
\end{equation}
where $\{\Omega_{X_z}\}_z$ is a nowhere-zero (holomorphic) local section of the Hodge bundle \cite{Tia}. 
We call $K^B$ the Weil--Petersson potential. 
For later use, we introduce the following useful formula for computing $K^B$.  

\begin{Prop}[\cite{FKY}] \label{WPformula}
Assume that there exist formal sums of Lagrangian submanifolds $\{L_i\}$ representing a basis of $H_3(X,\Z)/\mathrm{tor}(H_3(X,\Z))$. 
Then 
\begin{equation}
K^B(z)=- \log (\sqrt{-1} \sum_{i,j} \chi_\mathrm{Fuk}^{i,j}  \int_{L_i}\Omega_{X_z} \int_{L_j}\overline{\Omega_{X_z}}),  \label{WP potential2}  
\end{equation}
where $(\chi_\mathrm{Fuk}^{i,j})=(\chi_{\mathrm{Fuk}}(L_i,L_j))^{-1}$ is 
the inverse matrix for the Euler paring $\chi_\mathrm{Fuk}$ of the Fukaya category $\mathrm{D^bFuk}(X)$. 
\end{Prop}

We will develop the A-model Weil--Petersson geometry based upon this new expression (Equation (\ref{WP potential2})). 

\begin{Def}
Let $\gamma \in H_3(X,\Z)$ be a non-trivial $3$-cycle. 
Given an identification $H_3(X,\Z) \cong H_3(X_z,\Z)$, 
we define the normalized central charge of $\gamma$ by 
$$
Z(\Omega_{X_z},\gamma)= e^{\frac{K^B(z)}{2}}\int_\gamma \Omega_{X_z}, 
$$
where by abuse of notation $K^B(z)$ is given by the Equation (\ref{WP potential1}) (it depends not only on $z$ but also on $\Omega_{X_z}$). 
\end{Def}

Hereafter we always assume $\gamma \ne 0$. 
Let $\pi: \widetilde{\mathfrak{M}}_{\mathrm{Cpx}} \rightarrow \mathfrak{M}_{\mathrm{Cpx}}$ be the universal covering. 
Then $Z(-,\gamma)$ is a smooth function on the total space of the pullback $\pi^*\mathcal{L}$ of the Hodge bundle $\mathcal{L}$
$$
Z(-,\gamma): \pi^*\mathcal{L} \longrightarrow \C. 
$$ 
We observe that the absolute value $|Z(\Omega_{X_z},\gamma)|$ is independent of a choice of $\Omega_{X_z} \ne 0$. 
Therefore we obtain a function
$$
|Z(-,\gamma)|: \widetilde{\mathfrak{M}}_{\mathrm{Cpx}} \longrightarrow \R_{\ge 0}. 
$$
We call it the mass function of $\gamma \in H_3(X,\Z)$.

\begin{Thm} \label{attractor eq them}
A stationary point $z \in \widetilde{\mathfrak{M}}_{\mathrm{Cpx}}$ of the mass function 
$|Z(-,\gamma)|: \widetilde{\mathfrak{M}}_{\mathrm{Cpx}} \rightarrow \R_{\ge 0}$ 
with $Z(\Omega_{X_z},\gamma) \ne 0$ is characterized by the equation 
\begin{equation}
\gamma^{PD} = \mathrm{Re} (C\Omega_{X_z}), \ \ \ (\exists  C \in \C) \label{Attractor equation}
\end{equation}
in $H^3(X,\Z)$, where $\gamma^{PD}$ denotes the Poincar\'e dual of $\gamma$. 
We call the Equation (\ref{Attractor equation}) the complex attractor equation. 
It is equivalent to the condition 
$$
\gamma^{PD}  \in H^{3,0}(X) \oplus H^{0,3}(X). 
$$
\end{Thm}
\begin{proof} 
By the Bogomolov--Tian--Todorov Theorem, the Kodaira--Spencer map provides an identification between an open neighborhood $U$ of $z \in \widetilde{\mathfrak{M}}_{\mathrm{Cpx}}$ 
and an open neighborhood $U'$ of $0 \in H^1(X,TX)\cong H^{2,1}(X)$.  
Therefore, for a basis $\Omega_1,\dots,\Omega_k$ of $H^{2,1}(X)$, the variation $\Omega'_{X_z}=\Omega_{X_z}+\sum_i\epsilon_i \Omega_i$ gives a local coordinate $\epsilon=(\epsilon_1,\dots,\epsilon_n)$ of $U$. 
Then a straightforward calculation shows
\begin{align}
\left.\frac{\partial}{\partial \epsilon_i}\right|_{\epsilon=0} |Z(\Omega'_{X_z},\gamma)|^2
=& \left.\frac{\partial}{\partial \epsilon_i}\right|_{\epsilon=0} \frac{|\int_\gamma \Omega'_{X_z}|^2}{\sqrt{-1} \int_{X_z} \Omega'_{X_z} \wedge \overline{ \Omega'_{X_z}}} \notag \\
=& e^{K^B(z)} \int_\gamma \overline{\Omega_{X_z}} \int_\gamma \Omega_i  \notag 
\end{align}
By the assumption that $Z(\Omega_{X_z},\gamma)= e^{\frac{K^B(z)}{2}}\int_\gamma \Omega_{X_z} \ne 0$, 
$z \in \widetilde{\mathfrak{M}}_{\mathrm{Cpx}}$ is a stationary point if and only if $\int_\gamma \Omega_i=0 \ (1 \le i \le k)$. 
Since $\Omega_1,\dots,\Omega_k$ form a basis of $H^{2,1}(X)$, 
this condition is equivalent to 
$$
\gamma^{PD} \in H^{3,0}(X)\oplus H^{0,3}(X).
$$ 
Therefore, since $\gamma^{PD} \in H^3(X,\Z)$, 
$$
\gamma^{PD}=(C\Omega_{X_z}+\overline{ C\Omega_{X_z}})=\mathrm{Re} (C\Omega_{X_z})
$$
for some $C \in \C$. 
\end{proof}

\begin{Thm} \label{minimizer}
A stationary point $z \in  \widetilde{\mathfrak{M}}_{\mathrm{Cpx}}$ of the mass function $|Z(-,\gamma)|$ with $Z(\Omega_{X_z},\gamma) \ne 0$ is a local minimizer. 
Moreover, such points are discrete. 
\end{Thm}
\begin{proof}
For a stationary point $z \in \widetilde{\mathfrak{M}}_{\mathrm{Cpx}}$, a straightforward but tedious calculation shows
\begin{align}
\left.\frac{\partial^2}{\partial \epsilon_i\partial \overline{\epsilon_j}}\right|_{\epsilon=0}|Z(\Omega'_{X_z},\gamma)|^2
=&\left.\frac{\partial^2}{\partial \epsilon_i\partial \overline{\epsilon_j}}\right|_{\epsilon=0} \frac{|\int_\gamma \Omega'_{X_z}|^2}{\sqrt{-1} \int_{X_z} \Omega'_{X_z} \wedge \overline{ \Omega'_{X_z}}} \notag \\
=& \frac{- \sqrt{-1} |\int_\gamma \Omega_{X_z}|^2  \int_{X_z} \Omega_i \wedge \overline{\Omega_j} }{(\sqrt{-1} \int_{X_z} \Omega_{X_z} \wedge \overline{ \Omega_{X_z}})^2} \notag \\
=& 2e^{K^B(z)} |\int_\gamma \Omega_{X_z}|^2 g^B_{i \bar{j}}(z).  \notag 
\end{align}
Therefore the complex Hessian of the function $|Z(-,\gamma)|^2$ at $z$ is identified with the Weil--Petersson metric $g^B_{i \bar{j}}(z)$, rescaled by a positive constant, and hence is positive definite.  
\end{proof}

Theorem \ref{minimizer} asserts that there are 3 different types of the behavior of the mass function $|Z(-,\gamma)|$ 
depending on the nature of $\gamma \in H_3(X,\Z)$.  

\begin{enumerate}
\item There exists no stationary point. 
\item There exists a stationary point $z \in \widetilde{\mathfrak{M}}_{\mathrm{Cpx}}$ with $Z(\Omega_{X_z},\gamma)=0$. 
In this case the equation $\int_\gamma \Omega_{X_z}=0$ defines a divisor on $\widetilde{\mathfrak{M}}_{\mathrm{Cpx}}$.  
\item There exists a stationary point $z \in \widetilde{\mathfrak{M}}_{\mathrm{Cpx}}$ with $Z(\Omega_{X_z},\gamma) \ne 0$. 
\end{enumerate}

\begin{Def}
A stationary point $z \in  \widetilde{\mathfrak{M}}_{\mathrm{Cpx}}$ with $Z(\Omega_{X_z},\gamma) \ne 0$ is called a complex attractor for $\gamma$. 
The corresponding Calabi--Yau 3-fold $X_z$ is called a complex attractor variety for $\gamma$. 
We denote by $\widetilde{\Attr}_{\mathrm{Cpx}}(\gamma) \subset  \widetilde{\mathfrak{M}}_{\mathrm{Cpx}}$ the set of complex attractors for $\gamma$. 
Then we define 
$$
\Attr_{\mathrm{Cpx}}= \pi(\cup_{\gamma} \widetilde{\Attr}_{\mathrm{Cpx}}(\gamma)) \subset \mathfrak{M}_{\mathrm{Cpx}}, 
$$
where $\gamma$ runs over $H^3(X,\Z)$, and call it the complex attractor constellation of $X$. 
\end{Def}

\begin{Rem}
The complex attractor equation is concerned with $\Omega_{X_z}$. 
It in general does not determine a complex attractor $z \in  \widetilde{\mathfrak{M}}_{\mathrm{Cpx}}$ because a Torelli type theorem fails in 3-dimensions.  
\end{Rem}

It is natural to ask whether or not a complex attractor gives the global minimum, but the situation is rather complicated partly due to the non-compactness of the complex moduli space $\mathfrak{M}_{\mathrm{Cpx}}$.   
In fact, it is claimed in \cite[Section 9.2]{Moo} that there exists a Calabi--Yau 3-fold for which a single $\gamma$ leads to several distinct complex attractors with different values of local minima.  

Another important problem is to investigate the distribution of the complex attractor constellation $\Attr_{\mathrm{Cpx}} \subset \mathfrak{M}_{\mathrm{Cpx}}$. 
Note that it is not an intrinsic property of the complex manifold $\mathfrak{M}_{\mathrm{Cpx}}$. 
For example, many 1-parameter families of Calabi--Yau 3-folds share their complex moduli $\mathbb{P}^1\setminus\{0,1,\infty\}$,   
but the complex constellations should be different because they depend on the Calabi--Yau 3-folds they parametrize (to be precise, the variation of Hodge structures).

\begin{Conj} \label{conj lcsl}
Let $\mathfrak{M}_{\mathrm{Cpx}} \subset \overline{\mathfrak{M}}_{\mathrm{Cpx}}$ be a (partial) compactification. 
The complex attractor constellation $\Attr_{\mathrm{Cpx}}$ is dense near a large complex structure limit $z \in \overline{\mathfrak{M}}_{\mathrm{Cpx}}$. 
\end{Conj}

This conjecture is inspired by the observation that there seem infinitely many K\"ahler attractor points (to be introduced in Section \ref{Kahler attractor}) 
near the large volume limit of a Calabi--Yau 3-folds.

\begin{Rem}
A Calabi--Yau 3-fold $X$ is called rigid if $H^{2,1}(X)=0$. 
A rigid Calabi--Yau 3-fold $X$ is by definition a complex attractor variety for any $\gamma \in H_3(X,\Z)$.  
From this perspective, the complex attractor varieties are a vast generalization of the rigid Calabi--Yau 3-folds, whose arithmetic properties are of considerable interest.  
Indeed, in \cite[Section 8.2]{Moo}, Moore posed several interesting questions (the attractor conjectures) pertaining to the arithmetic nature of the complex attractor varieties. 
This direction of research has recently been carried out by Lam and Tripathy \cite{Lam, LT}. 
\end{Rem}

We now take a closer look at the complex attractors. 
The plane
$$
V(z) = H^{3,0}(X_z) \oplus H^{0,3}(X_z) \subset H^3(X, \C)
$$
varies as $z$ moves in $\widetilde{\mathfrak{M}}_{\mathrm{Cpx}}$, where we have a natural identification $H^3(X_z, \C) \cong H^3(X, \C)$ for a reference $X$. 
The intersection with the real $(2h^{2,1}+2)$-dimensional space $H^3(X, \R)$ is 
the 2-plane $V_\R(z)$ spanned over $\R$ by $\mathrm{Re}(\Omega_{X_z})$ and $\mathrm{Im}(\Omega_{X_z})$. 
For a generic $z \in \widetilde{\mathfrak{M}}_{\mathrm{Cpx}}$, the plane $V_\R(z)$ intersects $H^3(X,\Z) \subset H^3(X, \R)$ only in $0$.

\begin{Def}
Let $z \in \widetilde{\mathfrak{M}}_{\mathrm{Cpx}}$ be a complex attractor for some $\gamma' \in H_3(X,\Z)$. There are two cases: 
\begin{enumerate}
\item The intersection $V_\R(z) \cap H^3(X,\Z)$ is a lattice line. 
The point $z$ is a complex attractor for any non-zero $\gamma^{PD}  \in V_\R(z) \cap H^3(X,\Z)$. 
In this case $z$ is called a complex attractor of rank 1.
\item The intersection $V_\R(z) \cap H^3(X,\Z)$ is a lattice plane. 
Then there exist $\gamma_1, \gamma_2 \in H_3(X,\Z)$ such that the intersection $\gamma_1 \cap \gamma_2 \ne0$ 
and $V(z)$ is the complexification of the lattice $\Z \gamma_1^{PD} + \Z\gamma_2^{PD}$. 
Therefore $\gamma_1, \gamma_2$ simultaneously satisfy the complex attractor equations. In this case, $z$ is called a complex attractor of rank 2. 
\end{enumerate}
\end{Def}

\begin{Prop} \label{attractor rank 2}
Let $z\in \widetilde{\mathfrak{M}}_{\mathrm{Cpx}}$ be a complex attractor of rank 2 such that 
$\gamma_1, \gamma_2 \in H_3(X,\Z)$ with $\gamma_1 \cap \gamma_2 \ne0$ simultaneously satisfy the complex attractor equations
\begin{equation}
\gamma_1^{PD} = \mathrm{Re} (C_1\Omega_{X_z}), \ \ \ \gamma_2^{PD} = \mathrm{Re} (C_2\Omega_{X_z}). \label{2 attractor equations}
\end{equation}
for some $C_1, C_2 \in \C$. 
Then 
$$
\Omega_{X_z}=\frac{\sqrt{-1}}{\im(C_1\overline{C_2})}(\overline{C_1} \gamma_2^{PD}-\overline{C_2}\gamma_1^{PD}). 
$$
\end{Prop}
\begin{proof}
Since $\Omega_{X_z}$ lies in $V(z)$, we can write $\Omega_{X_z}=a_1 \gamma_1^{PD} + a_2 \gamma_2^{PD}$ for some $a_1,a_2 \in \C$. 
By plugging this in the complex attractor equations (\ref{2 attractor equations}), we determine the coefficients $a_1,a_2$. 
\end{proof}

While the complex attractors of rank 1 are expected to be dense in the moduli space, those of rank 2 are expected to be rare, 
as the underlying Calabi--Yau 3-fold in general need to satisfy very stringent conditions.
In fact, Moore showed by using mirror symmetry the complex attractors of rank 1 are dense near a large complex structure limit (Conjecture \ref{conj lcsl}) \cite{Moo}.


\subsection{Complex attractor mechanism for torus} \label{complex attractor mechanism for torus}

Let us consider a real $6$-dimensional torus $X=\C^3/(\Z^3 + \sqrt{-1} \Z^3)$. 
We introduce a complex structure on $X$ in such a way that 
$$
dz_i = dx_i + \sum_{j=1}^3T^{ij}dy_j \ \ \ (1 \le i \le 3)
$$
are holomorphic $1$-forms for a period matrix $T=(T^{ij}) \in \mathfrak{H}_3$. 
Such a complex torus is denoted by $X_T$. 
Then $X_T$ is biholomorphic to $\C^3/(\Z^3 + T \Z^3)$ equipped with the natural complex structure by the map
$$
\phi: \C^3/(\Z^3 \oplus \sqrt{-1} \Z^{3} )\longrightarrow  \C^3/(\Z^3 + T \Z^3), \ \ z=x+\sqrt{-1}y \mapsto z'=x+T y.
$$
We see $\phi$ is holomorphic because $\phi^*(dz')=d \phi^*(z')=dx+T dy$. 
We vary the complex structure of $X_T$ by, not varying the lattice as usual, but by varying the holomorphic volume form 
$$
\Omega_{X_T}=dz_1 \wedge dz_2 \wedge dz_3. 
$$
We fix a symplectic basis of $H^3(X_T,\Z)$ as follows
\begin{align}
\alpha_0 &= dx_1 \wedge dx_2 \wedge dx_3, \notag \\
\alpha_{ij} &= \frac{1}{2}\sum_{l,m=1}^3 \epsilon_{ilm} dx_l \wedge dx_m \wedge dy_j \ \  (1 \le i,j \le 3) \notag \\
\beta^0 &=- dy_1 \wedge dy_2 \wedge dy_3, \notag \\
\beta^{ij} &= \frac{1}{2}\sum_{l,m=1}^3 \epsilon_{jlm} dx_i \wedge dy_l \wedge dy_m \ \  (1 \le i,j \le 3).  \notag
\end{align}
where $\epsilon_{ilm}$ denotes the Levi--Civita symbol. 
With respect to this basis, $\Omega_{X_T}$ has an expansion 
$$
\Omega_{X_T}=\alpha_0+  \sum_{i,j=1}^3T^{ij}\alpha_{ij}+ \sum_{i,j=1}^3(\Cof(T)_{ij})\beta^{ij} -(\det(T)) \beta^0 
$$
where $\Cof(T)=(\Cof(T)_{ij})$ denotes the cofactor matrix of $T$. 
We fix a $3$-cycle $\gamma \in H_3(X_T,\Z)$ and write it as
$$
\gamma \ = q_0 A_0 + \sum_{i,j=1}^3Q_{ij} A_{ij}+ \sum_{i,j=1}^3P^{ij} B^{ij}+p^0 B^0
$$
where $A_0, A_{ij}, B^{ij}, B^0$ form a basis of $H_3(X_T,\Z)$ dual to the symplectic basis $\alpha_0, \alpha_{ij}, \beta^{ij}, \beta^0$. 
Then the normalized central charge of $\gamma$ reads
\begin{align}
Z(\Omega_{X_T},\gamma) & = e^{\frac{K^B(T)}{2}} \int_\gamma \Omega_{X_T} \notag \\
& =  e^{\frac{K^B(T)}{2}}(q_0 +  \sum_{i,j=1}^3Q_{ij}T^{ij} + \sum_{i,j=1}^3P^{ij}(\Cof(T)_{ij})-p^0\det(T)). \notag 
\end{align}

Therefore the complex attractor equation $\mathrm{Re}(C \Omega_{X_T})=\gamma^{PD}$ is equivalent to the following system of equations
\begin{align}
\mathrm{Re}(C) &= p^0 \notag \\
\mathrm{Re}(C T^{ij}) &= P^{ij}  \notag \\
\mathrm{Re}(C \Cof(T)_{ij}) &= -Q_{ij}  \notag \\
\mathrm{Re}(C \det (T)) &= q_0. \notag
\end{align}

\begin{Thm}[Moore \cite{Moo}]
\label{Thm:Moore}
A complex attractor for $\gamma \in H_3(X,\Z)$ exists if and only if the coefficient matrices $P=(P^{ij}), Q=(Q_{ij}) \in M_3(\Z)$ are symmetric. 
Then a complex attractor is unique and given by 
$$
T=((2PQ-(p^0q_0+\tr(PQ)E_3))+\sqrt{-D}E_3)(2R)^{-1} \in \mathfrak{H}_3
$$
where 
\begin{align}
R&=\Cof(P)+p^0Q, \notag \\
D&=((\tr(PQ))^2-\tr((PQ)^2))-(p^0q_0+\tr(PQ))^2+4(p^0 \det(Q) -q_0 \det(P)). \notag 
\end{align}
\end{Thm}
\begin{proof}
We provide in Appendix A a rigorous and accessible proof based on Moore's original argument.  
One of our contributions is to show that there is no complex attractor if $P,Q$ are not symmetric. 
\end{proof}

The complex attractor variety $X_T \cong \C^3/(\Z^3 + T \Z^3)$ has the following interesting property. 
The lattice embedding 
$$
\Z^3+T(2R)\Z^3 \hookrightarrow \Z^3 + T \Z^3
$$
induces an isogeny
$$
\phi: (E_{\sqrt{-D}})^3 \cong \C^3/(\Z^3+T(2R)\Z^3) \longrightarrow X_T. 
$$
In other words, the complex attractor variety $X_T$ is isogenous to 
the self-product $(E_{\sqrt{-D}})^3$ of the elliptic curves $E_{\sqrt{-D}}$ with complex multiplication 
by the covering map $\phi$ of degree $8 \det(R)$. 
In particular, $X_T$ is defined over a finite extension of the field $\Q(\sqrt{-D})$.

For a projective complex manifold $X$ the Lefschetz $(1,1)$-theorem asserts that the N\'eron--Severi group $NS(X)=H^2(X,\Z)\cap H^{1,1}(X)$. 
The rank $\rho(X)$ of the N\'eron--Severi group, the so-called Picard number, satisfies the inequality $1 \le \rho(X) \le h^{1,1}(X)$. 

\begin{Thm}[{\cite[Theorem 2.1]{HL}}] \label{maximal picard number}
Let $A$ be an abelian variety of dimension $g$. The following are equivalent. 
\begin{enumerate}
\item The Picard number is maximal, i.e. $\rho(A)=g^2$; 
\item $A$ is isogenous to the self-product of an elliptic curve $E$ with complex multiplication, i.e. $A \sim E^g$; 
\item $A$ is isomorphic to the product of some pairwise isogenous elliptic curves $E_1,\dots,E_g$ with complex multiplication, i.e. $A \cong E_1 \times \cdots \times E_g$. 
\end{enumerate}
\end{Thm}

Theorem \ref{maximal picard number} points out how the Picard number forces the structure of an abelian variety to be rigid. 
It is classically known that the algebraic varieties with the maximum Picard number possible often possess interesting arithmetic and geometric properties. 

\begin{Cor}
The complex attractor variety $X_T$ has the maximal Picard number $\rho(X_T)=9$, and hence is of rank 2. 
Moreover, $X_T \cong E_1 \times E_2 \times E_3$ for some pairwise isogenous elliptic curves $E_1, E_2,E_3$ with complex multiplication. 
\end{Cor}

In fact, the converse is also true and we have the following.

\begin{Thm} \label{attractor = 9}
The complex constellation $\Attr_{\mathrm{Cpx}}$ bijectively corresponds to the abelian 3-folds with Picard number $9$.  
\end{Thm}

\begin{proof}
It suffices to show that an abelian 3-fold $X_T = \C^3/(\Z^3 + T \Z^3)$ with $\rho(X_T)=9$ is a complex attractor variety for some $\gamma \in H_3(X_T,\Z)$.   
A proof is based on a straightforward but tedious computation, and we leave it in Appendix B. 
\end{proof}


\subsection{Complex attractor mechanism for $E \times S$} \label{complex attractor mechanism for  ExS}

Let $E=\C/(\Z + \sqrt{-1}\Z)$ be a real 2-dimensional torus. 
We put a complex structure on $E$ in such a way that $dz=dx+\tau dy$ is holomorphic for $\tau \in \HH$ so that $E \cong \C/(\Z + \tau\Z)$ as a complex manifold. 
Let $S$ be a K3 surface equipped with a holomorphic volume form $\Omega_S \in H^{2,0}(S)$. 
We consider the product Calabi--Yau 3-fold $X=E \times S$, which carries a natural holomorphic volume form 
$$
\Omega_X= dz \wedge \Omega_S. 
$$

Note that $dx,dy$ form a symplectic basis of $H^1(E,\Z)$. 
By the K\"unneth theorem, we have the identification $H^3(X,\Z)\cong H^1(E,\Z)\otimes_\Z H^2(S,\Z)$. 
Therefore the Poincar\'e dual of a 3-cycle $\gamma \in H_3(X,\Z)$ can be expressed as
$$
\gamma^{PD}=dx \otimes u_1+ dy \otimes u_2, 
$$
for some $u_1,u_2 \in H^2(S,\Z)$. 
We define $D_{u_1,u_2}=u_1^2u_2^2-(u_1,u_2)^2 \in \Z$. 

Then the complex attractor equation $\mathrm{Re}(C \Omega_X)=\gamma^{PD}$ is equivalent to the following system of equations
\begin{align}
\mathrm{Re}(C \Omega_S)=u_1 \notag \\
\mathrm{Re}(C\tau \Omega_S)=u_2 \notag 
\end{align}

Before solving the complex attractor equation, we introduce some notations. 
The N\'eron--Severi lattice of $S$ is $NS(S)=H^2(X,\Z) \cap H^{1,1}(S)$ equipped with the cup product.  
It is of signature $(1,\rho(S)-1)$. 
The transcendental lattice is its complement $T(S)=NS(S)^\perp \subset H^2(S,\Z)$. 
It is of signature $(2,20-\rho(S))$ and characterized as the minimal sublattice of $H^2(S,\Z)$ whose complexification contains $\Omega_S$.

\begin{Thm}[Moore \cite{Moo}] \label{K3 complex attractor}
A complex attractor for $\gamma \in H_3(X,\Z)$ exists if and only if the lattice $\Z u_1+ \Z u_2$ is positive definite.  
Moreover, if it exits, it is uniquely determined by the following periods
$$
\tau=\frac{(u_1,u_2) +\sqrt{-D_{u_1,u_2}}}{u_1^2}, \ \ \ 
\Omega_S=-\sqrt{-1}(\overline{\tau}u_1-u_2). 
$$
In particular, it is of rank 2. 
\end{Thm}
\begin{proof}
By Proposition \ref{attractor rank 2}, we obtain  
$$
\Omega_S=-\sqrt{-1}(\overline{\tau}u_1-u_2). 
$$
The global Torelli theorem asserts that the K3 surface $S$ is uniquely determined by $\Omega_S \in H^{2,0}(S)$, up to isomorphism.   
On the other hand, the Hodge--Riemann bilinear relation $\Omega_S \wedge \Omega_S=0$ implies
$$
u_1^2 \tau^2-2(u_1,u_2)\tau+u_2^2=0. 
$$
If $u_1^2=(u_1,u_2)=0$, then the transcendental lattice $T(S) \subset \Z u_1+ \Z u_2$ is degenerate and this is a contradiction. 
If $u_1^2=0$ and $(u_1,u_2) \ne 0$, then $\tau=\frac{u_2^2}{2(u_1,u_2)} \notin \HH$ and this is a contradiction. 
Hence $u_1^2 \ne 0$ and we get 
$$
\tau=\frac{(u_1,u_2)\pm\sqrt{-D_{u_1,u_2}}}{u_1^2}. 
$$ 
We need $D_{u_1,u_2}>0$ in order for $\tau$ to lie in $\HH$. 
If $u_1^2<0$, $T(S) \subset \Z u_1+ \Z u_2$ is negative definite and this is a contradiction. 
Therefore we conclude that $\Z u_1+ \Z u_2$ is positive definite.   
In this case, we have
$$
\tau=\frac{(u_1,u_2) +\sqrt{-D_{u_1,u_2}}}{u_1^2}
$$
Note that a change of the symplectic basis $dx,dy$ by an element of $\SL(2,\Z)$ does not change $\Omega_S$ and $\tau$. 
\end{proof}

At the complex attractor point, the N\'eron--Severi lattice is $NS(S)=(\Z u_1+ \Z u_2)^\perp$ and the Picard number $\rho(S)$ obtains its maximal possible value 20. 
A K3 surface with $\rho(S)=20$ is known as a singular K3 surface.  
It admits a rational map of degree 2 to a Kummer surface 
constructed from the product of two isogenous elliptic curves which have complex multiplications (a Shioda--Inose structure) \cite{SI}.  
Theorem \ref{K3 complex attractor} shows that any singular K3 surface appears as the second factor of some complex attractor variety $X=E \times S$.  

\begin{Rem}
A singular K3 surface is known as a rigid K3 surface as it does not admit any complex deformation keeping the property $\rho(S)=20$. 
\end{Rem}

\begin{Rem}
In the foundational article \cite{Dol}, Dolgachev formulated mirror symmetry for lattice polarized K3 surfaces. 
Although his formulation works beautifully in many case, a singular K3 surface has been an exception. 
The long-standing problem of mirror symmetry for singular K3 surfaces is recently settled in \cite{Kan} inspired by the results in this article (cf. Section \ref{Kahler constellation})
\end{Rem}

\begin{Thm}[\cite{SI}] 
There is a bijective correspondence between the isomorphism classes of singular K3 surfaces and the isomorphism classes of positive definite even lattices of rank 2.  
\end{Thm}

The bijective correspondence is given by associating a singular K3 surface $S$ with its transcendental lattice $T(S)$.

\begin{Lem} \label{dense tau}
Let $L$ be a positive definite even lattices of rank 2. Then 
$$
Q_L=\Big\{ \frac{(u_1,u_2) +\sqrt{-D_{u_1,u_2}}}{u_1^2} \Big\}_{u_1,u_2} \subset \HH
$$
is dense, where $u_1,u_2$ run over $\Z$-linearly independent vectors in $L$. 
\end{Lem}
\begin{proof}
Let us write $\tau_{u_1,u_2}=\frac{(u_1,u_2) +\sqrt{-D_{u_1,u_2}}}{u_1^2}$. 
Then for $k,l \in \Z$ we have the elementary identities: 
$$
\tau_{ku_1,lu_2}= \frac{l}{k}\tau_{u_1,u_2}, \ \ \ \tau_{u_1,ku_1+u_2}=k+\tau_{u_1,u_2}. 
$$
They show that $Q_L \subset \HH$ is dense. 
\end{proof}

\begin{Thm}
For $X=E \times S$, the complex constellation $\Attr_{\mathrm{Cpx}}$ is dense in the complex moduli space $\mathfrak{M}_{\mathrm{Cpx}}$. 
\end{Thm}
\begin{proof}
By a result of Beauville \cite{Bea}, we have a canonical isomorphism $\Aut(E\times S) = \Aut(E)\times \Aut(S)$. 
Then there is a natural fibration $p:\mathfrak{M}_{\mathrm{Cpx}}\rightarrow \mathfrak{M}_{\mathrm{Cpx}}^{K3}$, 
where $ \mathfrak{M}_{\mathrm{Cpx}}^{K3}$ denotes the complex moduli space of K3 surfaces.  
The image $p(\Attr_{\mathrm{Cpx}}) \subset  \mathfrak{M}_{\mathrm{Cpx}}^S$ corresponds to the isomorphism classes of singular K3 surfaces and is hence dense. 
Therefore it suffices to show that $p^{-1}(s)\cap \Attr_{\mathrm{Cpx}} \subset p^{-1}(s)\cong \mathrm{SL}(2,\Z) \backslash \HH$ is dense 
where $s$ corresponds to a singular K3 surface $S$. 
This assertion follows from the fact that $Q_{T(S)} \subset \HH$ is dense (Lemma \ref{dense tau}). 
\end{proof}

The results presented in Sections \ref{complex attractor mechanism for torus} and \ref{complex attractor mechanism for  ExS} 
imply that the complex attractors have maximal Picard numbers possible. 
This kind of results do not hold for the Calabi--Yau 3-folds with $h^{2,1}(X)=0$ as their Picard numbers are topological. 
Nevertheless, the complex attractors are discrete and the complex attractor varieties possess complex rigidity.


\section{Weil--Petersson geometry on Bridgeland stability space} \label{WP stability conditions}

We provide a brief review of our previous work \cite{FKY}. 
It introduced a provisional mirror Weil--Petersson geometry on the space of Bridgeland stability conditions $\Stab(\DD_X)$, 
which can be thought of as an approximation of the K\"ahler moduli space $\mathfrak{M}_{\mathrm{Kah}}$. 

\subsection{Bridgeland stability conditions} \label{Stab}
Let $X$ be a smooth projective variety of dimension $n$. 
We define $\DD_X=\mathrm{D^bCoh}(X)$ to be the bounded derived category of coherent sheaves on $X$. 

The numerical Grothendieck group $\NN(\DD_X)=K(\DD_X)/K(\DD_X)^{\perp_{\chi}}$ is the quotient group of the Grothendieck group $K(\DD_X)$ by the null group $K(\DD_X)^{\perp_{\chi}}$ of the Euler form $\chi$. 
It is a free abelian group is of rank $\sum_{i=0}^nh^{i,i}(X)$.  

\begin{Def}[\cite{Bri1}] 
A (Bridgeland) stability condition $\sigma=(\ZZ,\PP)$ on $\DD_X$ consists of a group homomorphism $\ZZ:\NN(\DD_X)\rightarrow\C$ 
and a collection $\PP=\{\PP(\phi)\}_{\phi \in \R}$ of full additive subcategories of $\DD_X$ parametrized by $\phi \in \R$ such that: 
\begin{enumerate}
\item If $0\neq F \in \PP(\phi)$, then $\ZZ(F)\in\R_{>0} \ e^{\sqrt{-1} \pi \phi}$.
\item $\PP(\phi+1)=\PP(\phi)[1]$. 
\item If $\phi_1>\phi_2$ and $A_i \in \PP(\phi_i)$, then $\Hom_{\DD_X}(A_1,A_2)=0$. 
\item For every $0 \ne F \in \mathcal{D}_X$, there exists a sequence of exact triangles
$$
\xymatrix{
0=F_0 \ar[r] & F_1\ar[d]  \ar[r] & F_2 \ar[r] \ar[d]& \cdots \ar[r]& F_{k-1}\ar[r] \ar[d] & F \ar[d]\\
                    & A_1 \ar@{-->}[lu]& A_2 \ar@{-->}[lu] &  & A_{k-1} \ar@{-->}[lu]  & A_k  \ar@{-->}[lu] 
}
$$
such that $A_i \in \PP(\phi_i)$ and $\phi_1>\phi_2>\cdots>\phi_k$.
\item (Support property \cite{KS}) There exist a constant $C>0$ and a norm $|\mkern-2mu|-|\mkern-2mu|$
on $\NN(\DD_X) \otimes_\Z \R$ such that $|\mkern-2mu|F|\mkern-2mu|\leq C|\ZZ(F)|$ for any semistable object $E$.
\end{enumerate}
$\ZZ$ is called a central charge and an element $A \in \PP(\phi)$ is called a semistable object of phase $\phi$. 
\end{Def}

We denote by $\Stab(\DD_X)$ the set of stability conditions on $\DD_X$. 
There is a nice topology on it such that the forgetful map
$$
\Stab(\DD_X)\longrightarrow \Hom(\NN(\DD_X),\C), \ \ \ \sigma=(\ZZ,\PP) \mapsto \ZZ
$$
is a local homeomorphism \cite{Bri1, KS}. 
In other words, the deformations of the central charge lift uniquely to deformations of the stability condition. 
Therefore $\Stab(\DD_X)$ naturally becomes a complex manifold, locally modelled on the $\C$-vector space $\Hom(\NN(\DD_X),\C) \cong \C^{\sum_{i=0}^nh^{i,i}(X)}$.  

Moreover, $\Stab(\DD_X)$ naturally carries a right action of the group $\widetilde{\GL}^+(2,\R)$, 
the universal covering of the group of orientation-preserving linear transformations $\GL^+(2,\R)$,  
as well as a left action of the group $\Aut(\DD_X)$ of autoequivalences of $\DD_X$. 
The $\widetilde{\GL}^+(2,\R)$-action is given by post-composition on the central charge $\ZZ:\NN(\DD_X)\rightarrow \C \cong \R^2$ with a suitable relabelling of the phases $\phi$. 
We often restrict this action to the subgroup $\C \subset\widetilde{ \GL}^+(2,\R)$ which acts freely.


\subsection{Central charge via twisted Mukai pairing}\label{phyexp}
The Mukai pairing on $H^*(X,\C)$ is defined, for $v,w\in H^*(X,\C)$, 
$$
\langle v,w\rangle=\int_X e^\frac{c_1(X)}{2} v^{\vee}  w, 
$$
where $v=\sum_j v_j\in\oplus_j H^j(X,\C)$ and its Mukai dual $v^{\vee}=\sum_j\sqrt{-1}^jv_j$. 
Note it differs from the usual Mukai pairing for K3 surfaces by a sign. 
We define a twisted Mukai vector of $F\in\DD_X$ by 
$$
v_\Lambda(F)=\ch(F)\sqrt{\Td_X} e^{\sqrt{-1}\Lambda}
$$
for any $\Lambda \in H^*(X,\C)$ such that $\Lambda^\vee=-\Lambda$. 
A twisted Mukai pairing is compatible with the Euler pairing; 
by the Hirzebruch--Riemann--Roch theorem, 
\begin{equation}
\chi(E,F)=\int_X \ch(E^\vee) \ch(F)\Td_X= \langle v_\Lambda(E),v_\Lambda(F)\rangle. \label{HRR}
\end{equation}
A geometric twisting $\Lambda_X$ compatible with the integral structure on the quantum cohomology was introduced by Iritani \cite{Iri} and Katzarkov--Kontsevich--Pantev \cite{KKP}. 

It is called the the log Gamma class and, in the Calabi--Yau case, we can explicitly write it as
$$
\Lambda_X=-\frac{\zeta(3)}{(2 \pi)^3}c_3(X)+\frac{\zeta(5)}{(2\pi)^5}(c_5(X)-c_2(X)c_3(X))+\dots
$$
For K3 and abelian surfaces, there is no effect of twisting as $\Lambda_X=0$. 
For Calabi--Yau 3-folds, the modification is given by the first term, which is familiar in the B-model period computations.   
\begin{Def}
We define $v_X(F)$ to be the twisted Mukai vector of $F \in \NN(\DD_X)$ associated to the log Gamma class $\Lambda_X$, namely 
$$
v_X(F)=\ch(F)\sqrt{\Td_X} e^{\sqrt{-1}\Lambda_X}
$$
\end{Def}

Let $X$ be a projective Calabi--Yau manifold equipped with a complexified K\"ahler parameter 
$$
\omega=B+\sqrt{-1} \kappa \in H^{1,1}(X,\C),
$$   
where $\kappa$ is a K\"ahler class. 
The set of such classes is the complexified K\"ahler cone $\mathcal{K}^\C_X$. 
We set $q=e^{2 \pi \sqrt{-1}\omega}$.

\begin{Conj}[Bridgeland \cite{Bri1}]  \label{Bridgeland Conj}
Let $\mathfrak{M}_{\mathrm{Cpx}}$ be the K\"ahler moduli space of a projective Calabi--Yau manifold $X$.  
Then there exists an embedding 
$$
\mathfrak{per}^\vee: \mathfrak{M}_{\mathrm{Kah}} \hookrightarrow \mathrm{Aut}(\DD_X) \backslash \Stab(\DD_X)/ \C.
$$
The complexified K\"ahler cone $\mathcal{K}_X^\C$ gives a local chart of $\mathfrak{M}_{\mathrm{Kah}}$ and, 
near the large volume limit, 
there exists a stability condition $\sigma_\omega$ with central charge of the form 
\begin{equation} \label{Central charge}
\ZZ_{\sigma_\omega}(F)=-\left \langle (2\pi \sqrt{-1})^{-\frac{\deg}{2}}J(-2\pi \sqrt{-1}\omega), v_X(F) \right\rangle. 
\end{equation}
Here $J(\tau)=J(\tau,1)$ denotes the $J$-function of $X$ evaluated at the spectral parameter $z=1$ 
and $\deg$ is the degree operator defined by $\deg(\alpha)=2p\alpha$ for $\alpha \in H^{p,p}(X)$. 
This expression was introduced by Iritani in his study of integral structures of quantum cohomology \cite{Iri}. 
$\ZZ_{\sigma_\omega}$ is called the quantum cohomology central charge (cf. Hosono \cite{Hos}).  
The embedding $\mathfrak{per}^\vee$ is comparable with a period map in mirror symmetry.  
\end{Conj} 
The asymptotic behavior of the quantum cohomology central charge $\ZZ_{\sigma_\omega}(F)$ is given by  
$$
\ZZ_{\sigma_\omega}(F) \sim -\int_X e^{- \omega} v_X(F)+ O(q). 
$$
The existence of a stability condition with the asymptotic central charge given by the leading term has been proven for various examples 
including K3 surfaces, abelian surfaces \cite{Bri2}, and abelian 3-folds \cite{BMS, MP}.


\subsection{Weil--Petersson geometry}
In this subsection we assume $X$ is a projective Calabi--Yau $n$-fold.  
Then the Serre duality implies that for $E, F \in \DD_X$, there is a natural functorial isomorphism
$$
\mathrm{Hom}^{*}_{\DD_X}(E,F)\cong\mathrm{Hom}^*_{\DD_X}(F,E[n])^{\vee}.
$$
An important consequence of the Calabi--Yau condition is that the Euler form on $\NN(\DD_X)$ is (skew-) symmetric if $n$ is even (odd). 
The following bilinear form is inspired by Proposition \ref{WPformula}. 

\begin{Def}
Let $\{F_i\}$ be a basis of  $\NN(\DD_X)$.  
We define a bilinear form $\mathfrak{b}: \Hom(\NN(\DD_X),\C)^2 \rightarrow \C$ by
$$
(\ZZ_1, \ZZ_2) \mapsto \sum_{i,j} \chi^{i,j}   \ZZ_1(F_i)\ZZ_2(F_j),  
$$
where $(\chi^{i,j})=(\chi(F_i,F_j))^{-1}$. 
Then $\mathfrak{b}$ is independent of the choice of a basis.  
\end{Def}
We think of $\Hom(\NN(\DD_X),\C)$ as the tangent space of $\Stab(\DD_X)$ at a point. 
Therefore $ \mathfrak{b}$ defines a holomorphic symplectic structure on $\Stab(\DD_X)$ for odd $n$.

\begin{Def} \label{Stab+}
We define $\Stab^+(\DD_X) \subset \Stab(\DD_X)$ by 
$$
\Stab^+(\DD_X)=\{ \sigma=(\ZZ,\PP) \ | \ \mathfrak{b}(\ZZ,\ZZ)=0, \ (\sqrt{-1})^{-n} \mathfrak{b}(\ZZ,\overline{\ZZ})>0\}.  
$$
If $n$ is odd, the first condition is vacuous as $\mathfrak{b}$ is skew-symmetric. 
\end{Def}

\begin{Rem}
It is worth mentioning that $\Stab^+(\DD_X)$ is an analogue of a period domain in the Hodge theory and the defining equations are an analogue of the Hodge--Riemann bilinear relations. 
The natural free $\C$-action on $\Stab(\DD_X)$ preserves $\Stab^+(\DD_X)$. 
\end{Rem}

\begin{Def}\label{WPpotential}
Let $s=(\ZZ_{\bar{\sigma}},\PP_{\bar{\sigma}})$ be a holomorphic local section of the $\C$-torsor $\Stab^+(\DD_X)\rightarrow\Stab^+(\DD_X)/\C$. 
Then 
\begin{equation}
K^A(\bar{\sigma})= -\log \left((\sqrt{-1})^{-n} \mathfrak{b}(\ZZ_{\bar{\sigma}},\overline{\ZZ_{\bar{\sigma}}}) \right) \label{AWP potential}
\end{equation}
defines a local smooth function on $\Stab^+(\DD_X)/\C$. 
We call $K^A$ the A-model Weil--Petersson potential.  
\end{Def}

\begin{Prop}[{\cite[Proposition 3.5]{FKY}}]
The complex Hessian $g^A=\frac{\sqrt{-1}}{2}\partial\overline{\partial}K^A$ of the A-model Weil--Petersson potential $K^A$ is independent of the choice of a local section $s$. 
Moreover, it descends to the quotient
$$
\Aut(\DD) \backslash\Stab^+(\DD_X)/ \C
$$
away from the singular loci. 
\end{Prop}

We call $g^A$ the A-model Weil--Petersson metric on $\Aut(\DD_X) \backslash\Stab^+(\DD_X)/ \C$. 
Note that $g^A$ is in general a degenerate metric. 
The following examples are discussed in \cite{FKY}.

\begin{Ex} \label{ell curve}
Let $X$ an elliptic curve. 
Since $\widetilde{\GL}^+(2,\R)$-action on $\Stab(\DD_X)$ is free and transitive, we observe
$$
\Stab^+(\DD_X)=\Stab(\DD_X)  \cong \widetilde{\GL}^+(2,\R) \cong \C \times \HH. 
$$
Therefore we conclude
$$
\Aut(\DD_X) \backslash\Stab^+(\DD_X)/ \C \cong \mathrm{SL}(2,\Z) \backslash \HH. 
$$
This is indeed the expected K\"ahler moduli space of $X$. 
Up to the $\C$-action, the central charge at $\tau\in\HH$ is given by
$$
\ZZ(F)=-\deg(F)+ \tau \ \rank(F).
$$ 
Since $K(\DD_X)=\Z \mathcal{O}_X \oplus \Z \mathcal{O}_p$, the A-model Weil--Petersson potential is 
\begin{align}
K^A(\tau)&=-\log\left((\sqrt{-1})^{-1}
(\mathcal{Z}(\mathcal{O}_p)\overline{\mathcal{Z}}(\mathcal{O}_X)-
\mathcal{Z}(\mathcal{O}_X)\overline{\mathcal{Z}}(\mathcal{O}_p))\right) \notag \\
&=-\log(\im(\tau))-\log2. \notag
\end{align}
This is the Poincar\'e potential on $\HH$ and descends to $\mathrm{SL}(2,\Z) \backslash \HH$. 
Therefore $g^A$ is the Poincar\'e metric. 
This computation is compatible with the fact that a mirror of an elliptic curve is an elliptic curve.
\end{Ex}

\begin{Ex}
Let $X$ be the self-product $E_\tau \times E_\tau$ of an elliptic curve $E_\tau=\C/(\Z + \tau \Z)$.  
Then there is an identification 
$$
\overline{\Aut}_\CY(\DD_X)\backslash\Stab^+(\DD_X)/\C^\times\cong
\mathrm{Sp}(4,\Z) \backslash \mathfrak{H}_2
$$
where $\overline{\Aut}_\CY(\DD_X)$ is an appropriate group induced by the group of autoequivalences.   
Moreover, the A-model Weil--Petersson metric on the LHS is identified with
the Bergman metric on the RHS (a Siegel modular variety).  
This result is compatible with the mirror symmetry between $X$ and a principally polarized abelian surface. 
The complex moduli space of the latter is given by $\mathrm{Sp}(4,\Z) \backslash \mathfrak{H}_2$. 
\end{Ex}

\begin{Ex}
Let $X \subset \mathbb{P}^4$ be a quintic Calabi--Yau 3-fold, for which the existence of a stability condition is recently proven by Li \cite{Li}. 
Let $\tau H \in H^2(X,\C)$ be the complexified K\"ahler class, where $H$ is the hyperplane class and $\tau \in \HH$. 
Then we have 
$$
J(2\pi  \sqrt{-1}\tau H,z)=e^{\frac{2\pi  \sqrt{-1}\tau H}{z}}
( 1 + \frac{1}{z^2} \sum_{d\ge0} N_d^X q^d d \frac{H^2}{5} - \frac{2}{z^3} \sum_{d \ge 0} N_d^X q^d \frac{H^3}{5} )
$$
where $q=e^{2\pi  \sqrt{-1} \tau}$ and $N_d^X$ denotes the genus 0 Gromov--Witten invariant of $X$ of degree $d$ (cf. \cite{CK}).  
Then we have 
\begin{align*}
J(2\pi  \sqrt{-1}\tau H) = & e^{2\pi  \sqrt{-1}\tau H} ( 1 +  \frac{1}{5} \sum_{d\ge0} N_d^X q^d d H^2 - \frac{2}{5} \sum_{d \ge 0} N_d^X q^d H^3 ) \\
 =&  1 + 2\pi  \sqrt{-1}\tau H + (\frac{1}{2}(2\pi  \sqrt{-1}\tau)^2 + \frac{1}{5} \sum_{d\ge0} N_d^X q^d d ) H^2 \\
 & + (\frac{1}{6}(2\pi  \sqrt{-1}\tau)^3 + \frac{1}{5} 2\pi  \sqrt{-1}\tau \sum_{d\ge0} N_d^X q^d d - \frac{2}{5} \sum_{d \ge 0} N_d^X q^d)H^3. 
\end{align*}
Therefore the quantum central charge reads
\begin{align*}
\ZZ_{\sigma_{\tau H}}(F)=&-\left \langle (2\pi \sqrt{-1})^{-\frac{\deg}{2}}J(-2\pi \sqrt{-1}\tau), v_X(F) \right \rangle \\
=& -\int_X e^{-\tau H}v_X(F) + 2 \ch_0(F) (\pi  \sqrt{-1}\tau \sum_{d\ge0} N_d^X q^d d + \sum_{d \ge 0} N_d^X q^d) \\
&+  \frac{1}{5} \sum_{d\ge0} N_d^X q^d d \int_X \ch_1(F) H^2
\end{align*}

Hence, near the large volume limit, the A-model Weil--Petersson potential is given by
\begin{align*}
K^A(\tau) =  -\log(\frac{2^3 \cdot 5}{3!} \im(\tau)^3) +O(q). \notag
\end{align*}
and hence $g^A$ is a quantum deformation of the Poincar\'e metric.  
In particular, for sufficiently small $q$, it is non-degenerate and the Weil--Petersson distance to the large volume limit is infinite. 
\end{Ex}


\section{K\"ahler attractor mechanism} \label{Kahler attractor}

We will introduce the provisional definition of the K\"ahler attractor mechanism mirror to the complex attractor mechanism. 
Throughout this section $X$ is a projective Calabi--Yau 3-fold. 
Let $k=\dim H^{1,1}(X)$ be the expected dimension of the complexified K\"ahler moduli space. 
We define $H^{ev}(X,\C)=\oplus_{i=0}^3H^{i,i}(X,\C)$.  


\subsection{K\"ahler attractor mechanism of stability space}

We first investigate some fundamental structures on $\Stab^+(\mathcal{D}_X)$.  

\begin{Def}
For $\sigma=(\mathcal{Z},\mathcal{P}) \in \Stab^+(\mathcal{D}_X)$, we define the normalized K\"ahler central charge of $F \in \NN(\DD_X)$ by 
$$
V(\sigma,F)= e^{\frac{K^A(\sigma)}{2}} \ZZ(F), 
$$
where by abuse of notation $K^A(\sigma)$ is given by the Equation (\ref{AWP potential}) ($K^A(\sigma)$ really depends on $\sigma$, not on the class $\bar{\sigma}$).   
\end{Def}

Then $V(-,F)$ defines a smooth function
$$
V(-,F): \Stab^+(\DD_X) \longrightarrow \C. 
$$
Since the absolute value $W(\sigma,F)$ is invariant under the $\C$-action, we obtain a function
$$
|V(-,F)|: \Stab^+(\DD_X)/\C \longrightarrow \R_{\ge 0}. 
$$
We call it the K\"ahler mass function of $F \in \NN(\DD_X)$.  

Let us recall some basic facts about $\Stab^+(\DD_X)/\C$, a projectivization of a holomorphic symplectic manifold $\Stab^+(\DD_X)$. 
The tangent space of $\Stab^+(\DD_X)/\C$ at $\sigma=(\ZZ,\PP)$ is naturally identified with $ \Hom(\NN(\DD_X),\C)/\C\ZZ$, 
and hence $\Stab^+(\DD_X)/\C$ is a holomorphic contact manifold with the canonical contact form 
$$
\theta= \mathfrak{b}(d\ZZ,\ZZ). 
$$

There is a precise relation between Legendrian and Lagrangian submanifolds: the lift of a Legendrian submanifold in $\Stab^+(\DD_X)/\C$ is a Lagrangian submanifold in $\Stab^+(\DD_X)$. 
Moreover, a Legendrian subspace $L$ defines a polarized Hodge structure of weight 3 on $ \Hom(\NN(\DD_X),\Z)$ equipped with $\mathfrak{b}$, 
namely 
$$
H^{3,0}=\C \ZZ, \ \ \ H^{3,0} \oplus H^{2,1}=\pi^{-1}(L),
$$
where $\pi:\Stab^+(\DD_X) \rightarrow \Stab^+(\DD_X)/\C$ is the quotient map.  

\begin{Thm} \label{Kahler attractor on stab}
Let $F \in \NN(\DD_X)$ and $\sigma=(\mathcal{Z},\mathcal{P})\in\Stab^+(\mathcal{D}_X)/\C$ such that $\mathcal{Z}(F)\neq 0$. 
Given a Legendrian subspace $L\subset \Hom(\NN(\DD_X),\C)/\C\ZZ$, 
then $\sigma$ is a stationary point of the K\"ahler mass function $|V(-,F)|$ along the $L$-direction if and only if 
$$
\chi(F,-) = \mathrm{Re}(C \ZZ(-)) \ \ \ (\exists  C \in \C)
$$
holds in $\Hom(\NN(\DD_X),\C)$. 
\end{Thm}

\begin{proof}
The proof is parallel to that of Theorem \ref{attractor eq them}. 
Let $L_1,\ldots,L_{k} \in \Hom(\NN(\DD_X),\C)$ be lifts of a basis of $L$.  
By Bridgeland's result \cite{Bri1}, the deformation of the central charge 
$$
\mathcal{Z}_\epsilon = \mathcal{Z}+\sum_{i=1}^{k}\epsilon_i L_i \in \Hom(\NN(\DD_X),\C)
$$ 
for small $\epsilon=(\epsilon_i)_{i=1}^{k} \in \C^{k}$ 
induces a unique deformation $\sigma_\epsilon$ of the stability condition $\sigma$.  
Then a straightforward calculation shows 
$$
\left.\frac{\partial}{\partial \epsilon_i}\right|_{\epsilon=0} |V(\sigma_\epsilon,F)|^2=e^{K^A(\sigma)} \overline{\ZZ(F)}L_i(F). 
$$
By the assumption that $\mathcal{Z}(F)\neq 0$, $\sigma$ is a stationary point of $|V(-,E)|^2$ along the $L$-direction if and only if $L_i(F)=0$ for $1 \le i \le k$.  
Then the Legendrian property of $L$ implies that $\chi(F,-) \in \C \ZZ \oplus \C \overline{\ZZ}$. 
Since $\chi(E,-) \in \Hom(\NN(\DD_X),\Z)$, we have 
$$
\chi(F,-) = (C \ZZ + \overline{C \ZZ}) =  \mathrm{Re}(C \ZZ(-)) 
$$
for some $C \in \C$. 
\end{proof}

%

\subsection{K\"ahler attractor mechanism of complexified K\"ahler cone}
We defined the normalized K\"ahler central charges on $\Stab^+(\DD_X)/\C$. 
However, $\Stab^+(\DD_X)/\C$ is in general conjectured to be much larger than (the universal covering of) the K\"ahler moduli space $\mathfrak{M}_{\mathrm{Kah}}$ (Conjecture \ref{Bridgeland Conj}).   
Therefore it is reasonable to restrict ourselves to a submanifold of the expected dimension $k$. 
In this section, we will consider the complexified K\"ahler cone $\mathcal{K}^\C_X$, a natural candidate of such a submanifold.  

Let $\omega=B+\sqrt{-1}\kappa \in \mathcal{K}^\C_X$ be a complexified K\"ahler class of $X$. 
Henceforth we consider the quantum cohomology central charge (Equation (\ref{Central charge}))
$$
\ZZ_{\sigma_\omega}(F)= - \left \langle \widetilde{J}(\omega), v_X(F) \right\rangle. 
$$
where we write $\widetilde{J}(\omega)= (2\pi \sqrt{-1})^{-\frac{\deg}{2}}J(-2\pi \sqrt{-1} \omega)$ for the sake of shorthand. 

Let $\phi_1,\dots,\phi_k \in H^{1,1}(X)$ be a basis and $t_1,\dots,t_k$ the linear coordinate system of $H^{1,1}(X)$ dual to the basis, i.e. we may write $\omega=\sum_{i=1}^k t_i\phi_i$. 
Let $L(\omega)$ be the fundamental solution of the quantum differential equation, that is the $\mathrm{End}(H^{ev}(X,\C))$-valued function 
satisfying
$$
\nabla^A L(\omega)=0, \ \ \ L(\omega) =\mathrm{id} + O(\omega),
$$
where $\nabla^A=d+\sum_{i=1}^k(\phi_i *) dt_i$ denotes the Dubrovin connection on $H^{ev}(X,\C)$.  
Then the $J$-function is obtained by applying the fundamental solution $L(\omega)$ to $1 \in H^0(X,\Z)$, i.e. $J(\omega)=L(\omega)1$. 

\begin{Prop}
Near the large volume limit (i.e. for sufficiently small $q$), $\sqrt{-1} \mathfrak{b}(\ZZ_{\sigma_\omega},\overline{\ZZ_{\sigma_\omega}})>0$ holds. 
\end{Prop}
\begin{proof}
This  is a quantum corrected version of \cite[Proposition 4.6]{FKY}, where $\widetilde{J}(\omega)$ is replaced by $e^\omega$. 
Since $\widetilde{J}(\omega) = e^\omega + O(q)$, the assertion follows.  
\end{proof}

\begin{Thm} \label{Legendre immersion}
The conjectural embedding (Conjecture \ref{Bridgeland Conj})
$$
\iota:  \mathcal{K}^\C_X \longrightarrow \Stab^+(\DD_X)/\C, \ \ \  \omega \mapsto \sigma_\omega=(\ZZ_{\sigma_\omega}, \PP_{\sigma_\omega})
$$
is Legendrian (if it exists).  
\end{Thm}
\begin{proof}
It suffices to check that 
$$
\C\langle \frac{\partial}{\partial t_1}\ZZ_{\sigma_{\omega}} ,\dots, \frac{\partial}{\partial t_k}\ZZ_{\sigma_{\omega}}  \rangle 
\subset \Hom(\NN(\DD_X),\C)/\C\ZZ_{\sigma_{\omega}}
$$
is a Legendrian subspace. 
This follows from the following calculation. 
\begin{align*}
\mathfrak{b}(\frac{\partial}{\partial t_i}\ZZ_{\sigma_{\omega}}, \frac{\partial}{\partial t_j}\ZZ_{\sigma_{\omega}}) 
 & =
 -\langle \frac{\partial}{\partial t_i} \widetilde{J}(\omega), \frac{\partial}{\partial t_j}\widetilde{J}(\omega) \rangle  \\
  & = 
   \frac{-1}{(-2\pi \sqrt{-1})^{3}} \langle \frac{\partial}{\partial t_i} J(-2\pi \sqrt{-1}\omega), \frac{\partial}{\partial t_j} J(-2\pi \sqrt{-1}\omega) \rangle  \\
 &=
   \frac{1}{2\pi \sqrt{-1}}  \langle L(-2\pi \sqrt{-1}\omega) \phi_i, L(-2\pi \sqrt{-1}\omega)\phi_j \rangle  \\
       &=
   \frac{1}{2\pi \sqrt{-1}}  \langle  \phi_i, \phi_j \rangle  \\
  & = 0
\end{align*}
The first equality is due to \cite[Lemma 4.4]{FKY}. 
Although \cite[Lemma 4.4]{FKY} is classical (no quantum correction), an identical proof works. 
The third equality follows from the definition of the $J$-function. 
The fourth equality follows from \cite[Proposition 4.2]{Iri} (essentially the Frobenius property of the quantum product).  
\end{proof}

Motivated by the Theorem \ref{Legendre immersion}, we define the normalized K\"ahler central charge of $F \in \NN(\DD_X)$ on $\mathcal{K}^\C_X$ by 
$$
W(\omega,F)= e^{\frac{K^A(\omega)}{2}} \ZZ_{\sigma_\omega}(F), 
$$
where $K^A$ denotes the A-model Weil--Petersson potential.   
Then $W(-,F)$ is a smooth function 
$$
W(-,F): \mathcal{K}^\C_X \longrightarrow \C
$$
and the K\"ahler mass function of $F \in \NN(\DD_X)$ is defined by
$$
|W(-,F)|: \mathcal{K}^\C_X \longrightarrow \R_{\ge 0}. 
$$ 
To summarize, we have the following commutative diagram
$$
\xymatrix{
\mathcal{K}^\C_X  \ar[rd]_{|W(-,F)|}   \ar@{^{(}->}[r]^-{\iota} & \Stab^+(\DD_X)/\C \ar[d]^{|V(-,F)|} \\
 &  \R_{\ge 0}
}
$$
where $\iota$ is in general hypothetical, but the K\"ahler mass functions are well-defined. 

\begin{Thm}
For $F \in \NN(\DD_X)$, a stationary point $\omega \in \mathcal{K}^\C_X$ of the K\"ahler mass function such that $W(\omega,F) \ne 0$ is characterized by the equation
\begin{equation}
\chi(F,-) = \mathrm{Re}(C \ZZ_{\sigma_\omega}(-)), \ \ \ (\exists C \in \C)\label{Attractor equation 2}
\end{equation}
in $\Hom(\NN(\DD_X),\C)$. 
We call Equation (\ref{Attractor equation 2}) the K\"ahler attractor equation. 
\end{Thm}
\begin{proof}
The assertion follows from Theorem \ref{Kahler attractor on stab} and  Theorem \ref{Legendre immersion}. 
Note that we do not used $\PP_{\sigma_\omega}$ but only $\ZZ_{\sigma_\omega}$ in Theorem \ref{Legendre immersion}.  
\end{proof}

\begin{Def}
A stationary point $\omega \in \mathcal{K}^\C_X$ with $W(\omega,F) \ne 0$ is called a K\"ahler attractor for $F$. 
The corresponding Calabi--Yau 3-fold $(X,\omega)$ is called a K\"ahler attractor variety for $F$.  
\end{Def}

\begin{Rem}
Near the large volume limit, we have an asymptotic expansion $K^A(\omega) = -\log \im (\omega)^3 + O(q)$ (up to constant term). 
Then the A-model Weil--Petersson metric $g^A$ is positive definite and we can show that the K\"ahler attractors are discrete by an almost identical argument to the complex side (Theorem \ref{minimizer}).  
Note that this sort of asymptotic metric has previously been investigated by Trenner and Wilson \cite{TW}. 
Our work \cite{FKY} can be considered as a globalization of their pioneering work. 
\end{Rem}


\subsection{K\"ahler attractor mechanism for torus} \label{Kahler attractor for T6}

Let us consider a complex 3-torus $Y=\C^3/(\Z^3+\sqrt{-1}\Z^3)$. 
We choose a symplectic basis of $H^{ev}(Y,\Z)=\oplus_{i=0}^3H^{i,i}(Y,\Z)$ as follows
\begin{align}
\delta_0 &= 1 \ \in H^{0,0}(Y,\Z), \notag \\ 
\delta_{ij} &= \frac{\sqrt{-1}}{2}dz_i \wedge d\bar{z_j} \ \in H^{1,1}(Y,\Z)\ \ \ (1 \le i,j \le 3) \notag \\
\epsilon^0 &= (\frac{\sqrt{-1}}{2})^3 dz_1 \wedge d\bar{z_1} \wedge dz_2 \wedge d\bar{z_2} \wedge dz_3 \wedge d\bar{z_3} \in H^{3,3}(Y,\Z) \notag   \\
\epsilon^{ij} &= (\frac{\sqrt{-1}}{2})^{-1}\frac{\partial}{\partial z_i}\lrcorner(\frac{\partial}{\partial \bar{z_j}}\lrcorner \epsilon^0) \in H^{2,2}(Y,\Z) \ \ \ (1 \le i,j \le 3).  \notag 
\end{align}
We introduce a complexified K\"ahler structure on $Y$ by
$$
\omega = B+ \sqrt{-1} \kappa = \sum_{i,j=1}^3 \omega^{ij} \delta_{ij} \in H^{1,1}(Y,\C). 
$$
and identify $\omega$ with the matrix $\Omega=(\omega^{ij}) \in \mathfrak{H}_3$ for the sake of convenience.

The twisted Mukai vector of $F \in \NN(\DD_Y)$ has an expansion
$$
v_Y(F)= v^0 \delta_0 + \sum_{i,j}^3 v^{ij}\delta_{ij} +  \sum_{i,j}^3 u_{ij}\epsilon^{ij} + u_0 \epsilon^0. 
$$
There is no quantum correction ($\ZZ_{\sigma_\omega}(F)=-\langle e^\omega, F \rangle$), 
and hence the K\"ahler attractor equation $\chi(F,-) = \mathrm{Re}(C \ZZ_{\sigma_\omega}(-))$ is equivalent to the following system of equations 
\begin{align}
\mathrm{Re}(C) &= -v^0 \notag \\
\mathrm{Re}(C \omega^{ij}) &=- v^{ij}  \notag \\
\mathrm{Re}(C \Cof(\Omega)_{ij}) &= u_{ij}  \notag \\
\mathrm{Re}(C \det(\Omega)) &= -u_0. \notag
\end{align}

They are parallel to the complex attractor equation for $T^6$. 
We are able to solve the equations to obtain the solutions in an explicit form. 

\begin{Thm}\label{thm:KahlerAbelThree}
Assume that the coefficient matrices $V=(v^{ij}), U=(u_{ij}) \in M_3(\Z)$ are symmetric. 
There exists a unique K\"ahler attractor 
$$
\Omega=((2VU-(v^0u_0+\tr(VU)E_3))+\sqrt{-D}E_3)(2R)^{-1} \in \mathfrak{H}_3
$$
where 
\begin{align}
R&=\Cof(V)+v^0U, \notag \\
D&=((\tr(VU))^2-\tr((VU)^2))-(v^0u_0+\tr(VU))^2+4(v^0 \det(U) -u_0 \det(V)). \notag 
\end{align}
\end{Thm}

\begin{proof}
The proof is based on a step-by-step explicit calculation given in Appendix A and is parallel to the complex attractor case. 
\end{proof}

In light of the B-model side, we introduce a covering of the K\"ahler attractor variety $(Y,\omega)$.  
The lattice embedding 
$$
\Z^3+\sqrt{-1}(2R)\Z^3 \hookrightarrow \Z^3 + \sqrt{-1} \Z^3
$$
induces a covering map 
$$
\phi^\vee: Y'= \C^3/(\Z^3+\sqrt{-1}(2R)\Z^3) \longrightarrow Y=\C^3(\Z^3 + \sqrt{-1} \Z^3). 
$$
Then $Y'$ carries a natural complexified K\"ahler structure $\omega'$ given by the pullback 
$$
\Omega'=(\phi^\vee)^*\Omega=(2VU-(v^0u_0+\tr(VU)E_3))+\frac{\sqrt{-D}}{2}E_3
$$
If we regard the complexified K\"ahler structures as elements of $H^{2}(Y',\C)/H^{2}(Y',\Z)$, 
then the B-field $\mathrm{Re}(\omega')$ becomes trivial and the K\"ahler structure reads
$$
\mathrm{Im}(\omega')=\frac{\sqrt{D}}{2}\sum_{1\le i,j\le 3} \delta_{ij} dz_i \wedge d\bar{z_j} = \sqrt{D} \sum_{1\le i \le 3} dx_i \wedge dy_i
$$
Hence $Y'$ is a principally polarized abelian 3-fold with the K\"ahler structure $\sqrt{D}$. 
This computation is compatible with the fact that the product $X'=(E_{\sqrt{-D}})^3$  of an elliptic curve $E_{\sqrt{-D}}$ 
is mirror symmetric to a principally polarized abelian 3-fold $Y'$ (c.f. \cite{KL}).  

We conclude that a complex attractor variety $X$ (resp. $X'$) is mirror symmetric to a K\"ahler attractor variety $Y$ (resp. $Y'$) 
provided that $\gamma \in H_3(X,\Z)$ and $F \in \NN(\DD_Y)$ are mirror cycles.  
 $$
\xymatrix{
(E_{\sqrt{-D}})^3 \  \ar[d]_\phi \ar@{<->}[rr]^-{\text{mirror}} & & \  (\C^3/(\Z^3+\sqrt{-1}(2R)\Z^3),\omega')  \ar[d]^{\phi^\vee} \\
 \C/(\Z^3+T \Z^3)                    & &  (\C^3/(\Z^3+\sqrt{-1}\Z^3),\omega) 
}
$$

\begin{Rem}
It is expected that if $\Omega$ is a complex attractor of a charge $\gamma$, then $\gamma$ supports BPS states with respect to $\Omega$ \cite[Section~2.6]{Moo}.
In terms of K\"ahler attractors and stability conditions on the mirror side, one expects that if $\sigma$ is a K\"ahler attractor of a class $v$, then $v$ should support a Bridgeland semistable object with repsect to $\sigma$.
The space of stability conditions on abelian threefolds has been studied in \cite{BMS}.
It would be interesting to verify that the K\"ahler attractors found in Theorem~\ref{thm:KahlerAbelThree} do indeed support semistable objects.
\end{Rem}


\subsection{K\"ahler attractor mechanism for $E \times S$} \label{Kahler attractor for E x S}
For an elliptic curve $E=\C/(\Z+\tau \Z)$ and a K3 surface $S$, we consider the product Calabi--Yau 3-fold $Y=S \times E$. 
In the following, we use $\left<-,-\right>$ for the Mukai pairing on $H^*(Y,\Z)$, $\left<-,-\right>_S$ for the Mukai pairing on $H^*(S,\Z)$, 
and $(-,-)$ for the cup pairing on $H^2(S,\Z)$. 
We also define the algebraic lattice by 
$$
NS'(S)=H^{0}(S,\Z) \oplus NS(S) \oplus H^{4}(S,\Z). 
$$

The twisted Mukai vector $v_Y(F)$ of $F \in \NN(\DD_Y)$ can be written as
\[
v_Y(F)=v_1+v_2 \frac{\sqrt{-1}}{2}dz\wedge d\bar z\in H^{ev}(Y,\Z)
\]
where for $i=1,2$ 
$$
v_i=(r_i,D_i,s_i)\in NS'(S). 
$$ 
We would like to find $\omega_S\in NS(S)_\C$ and $\omega_E\in\HH$ such that 
$$
\omega=\omega_S+\omega_E  \frac{\sqrt{-1}}{2} dz\wedge d\bar z\in H^{1,1}(Y)
$$
satisfies the K\"ahler attractor equation
\[
\re(CZ_\omega(-))=\left<v_Y(F),-\right> 
\]
for some $C\in\C$, where we rewrite the equation by the Hirzebruch--Riemann--Roch theorem (Equation (\ref{HRR})). 

\begin{Lem}
The K\"ahler attractor equation is equivalent to the following system of equations
\begin{align}
\re(C\left<\delta,-\right>_S) &=\left<v_1, -\right>_S,  \notag \\
\re(C\omega_E\left<\delta,-\right>_S) & =\left<v_2,-\right>_S, \notag 
\end{align}
where $\delta = e^{\omega_S}$.  
\end{Lem}
\begin{proof}
Recall first that there is no quantum correction for $Y$, and hence we have $\widetilde{J}(\omega)=e^\omega$. 
We plug $\alpha\in\oplus_{i=0}^2 H^{i,i}(S)$ in the K\"ahler attractor equation to obtain 
$$
\left<v_2,\alpha\right>_S=\left< v_Y(F),\alpha \right>=\re(C\left<e^\omega,\alpha\right>)=\re(C\omega_E\left<\delta,\alpha\right>).
$$
Similarly for $\beta dz\wedge d\bar z$ where $\beta\in\oplus_{i=0}^2 H^{i,i}(S)$ we obtain obtain 
$$
\left<v_1,\beta \right>_S=\left<v_Y(F),\beta \right>=\re(C\left<e^\omega,\beta \right>)=\re(C\left<\delta,\beta \right>).
$$
\end{proof}

Let us write $D_{v_1,v_2}=v_1^2v_2^2-\langle v_1,v_2\rangle^2_S \in \Z$, where $v_i^2=\langle v_i,v_i\rangle_S$.

\begin{Prop} \label{Kahler attractor 1}
There exist $\omega_S \in H^{1,1}(S)$ and $\omega_E \in \HH$ satisfying the K\"ahler attractor equation for $F$ 
if and only if the lattice $\Z v_1+ \Z v_2$ is positive definite.  
Moreover, they are unique and given by
$$
\delta=\frac{- \sqrt{-1}}{C \ \im(\omega_E)}(v_2-\overline{\omega}_Ev_1), \ \ \ 
\omega_E=\frac{\left<v_1,v_2\right>_S + \sqrt{-D_{v_1,v_2}}}{v_1^2}. 
$$
where the constant $C$ is chosen so that the degree 0 part of $\delta$ is $1 \in H^0(S,\C)$.   
\end{Prop}

\begin{proof}
First, by Proposition \ref{attractor rank 2}, we can solve the K\"ahler attractor equation to obtain 
$$
\delta
=\frac{\sqrt{-1}}{\im(|C|^2 \overline{\omega}_E)}(\overline{C}v_2-\overline{C \omega_E}v_1)
=\frac{- \sqrt{-1}}{C \ \im(\omega_E)}(v_2-\overline{\omega}_Ev_1).
$$
Assume that there exist $\omega_S \in H^{1,1}(S)$ and $\omega_E \in \HH$ satisfying the K\"ahler attractor equation. 
The condition that $\delta$ is of the form $\delta=e^{\omega_S}$ implies that $\left<\delta,\delta\right>=0$ and hence
$$
\omega_E^2v_1^2-2 \omega_E\left<v_1,v_2\right>_S+v_2^2 = 0. 
$$
By a parallel argument to the proof of Proposition \ref{K3 complex attractor}, we conclude that 
the lattice $\Z v_1 + \Z v_2$ is positive definite and 
$$
\omega_E=\frac{\left<v_1,v_2\right>_S + \sqrt{-D_{v_1,v_2}}}{v_1^2} \in \HH. 
$$
On the other hand, assume that $\Z v_1 + \Z v_2$ is positive definite, then $(r_1,r_2) \ne (0,0)$.  
Moreover, we can choose a constant $C$ so that the degree 0 part of 
$$
\delta=\frac{- \sqrt{-1}}{C \ \im(\omega_E)}(v_2-\overline{\omega}_Ev_1)
$$
is $1$, and then $\delta=e^{\omega_S}$ can be solved for $\omega_S=\log \delta  \in H^{1,1}(S)$. 
$\omega_E$ is uniquely determined in a similar manner. 
\end{proof}

Note that $\omega_S \in H^{1,1}(S)$ in Proposition \ref{Kahler attractor 1} is not necessarily a complexified K\"ahler class. 
The best we could prove is the following (cf. Remark \ref{positivity by gCY}).

\begin{Prop} \label{positive cone}
Assume the lattice $\Z v_1+ \Z v_2$ is positive definite. 
For $\omega_S=\log \delta \in H^{1,1}(S)$, we have $ \im(\omega_S)^2>0$.  
\end{Prop}
\begin{proof}
Recall that we write $v_i=(r_i,D_i,s_i) \in NS'(S)$. 
A straightforward computation of shows that $\im(\omega_S)^2=(r_2D_1-r_1D_2)^2$. 
Since $\Z v_1 + \Z v_2$ is positive definite, $(r_1,r_2) \ne (0,0)$. 
Assume first that $r_1r_2 \ne 0$. 
Then, since $\Z v_1+\Z v_2$ is positive definite
$$
0<(\frac{1}{r_1}v_1-\frac{1}{r_2}v_2)^2=(0,\frac{1}{r_1}D_1-\frac{1}{r_2}D_2,\frac{s_1}{r_1}-\frac{s_2}{r_2})^2=(\frac{1}{r_1}D_1-\frac{1}{r_2}D_2)^2. 
$$
Therefore we have
$$
\im(\omega_S)^2=(r_1r_2)^2(\frac{1}{r_1}D_1-\frac{1}{r_2}D_2)^2>0. 
$$
Assume next that $r_1=0$ and $r_2 \ne 0$.  
Then we have 
$$
\im(\omega_S)^2=r_2^2 D_1^2=r_2^2 v_1^2>0. 
$$
Similarly for $r_1 \ne 0$ and $r_2 = 0$. 
\end{proof}

\begin{Ex} \label{rigid Kahler}
Let $S$ be a K3 surface such that $NS(S)=\Z H$, where $H$ is ample with $H^2=2n>0$. 
It is known that for $v \in NS'(S)$ with $v^2 >0$ there is a sheaf $\mathcal{E}$ such that $v_S(\mathcal{E})=v$. 
In particular there are sheaves $\mathcal{E}_1,\mathcal{E}_2$ whose twisted Mukai vectors are 
$$
v_1=(1,0,-n), \ v_2=(0,-H,0) \in NS'(S). 
$$ 
For $Y=S \times E$, let us consider
$$
F=\mathcal{E}_1 \boxtimes \mathcal{O}_E \oplus \mathcal{E}_2 \boxtimes \mathcal{O}_p \in \NN(Y). 
$$
Then
$$
e^{\sqrt{-1}H} \in (\Z v_1+\Z v_2)_\C, \ \ \ \omega_E=\sqrt{-1} \in \HH
$$
satisfy the K\"ahler attractor equation for $F$.  
\end{Ex}

Let us take a close look at Example \ref{rigid Kahler}.  
\begin{align}
e^{\sqrt{-1}H} &= (1,\sqrt{-1}H,-n) \notag \\
&=v_1+\sqrt{-1}v_2 \in (\Z v_1 + \Z v_2)_\C \subsetneq NS'(S)_\C \notag 
\end{align}
On the other hand, for $\epsilon^2 \notin \Q$, 
\begin{align}
e^{\sqrt{-1}\epsilon H} &=(1,\sqrt{-1} \epsilon H,-\epsilon^2n) \notag \\
&= (1,0,-\epsilon^2n) + \sqrt{-1}\epsilon(0, H, 0) \notag \\
&= (1,0,0) -\epsilon^2 (0,0,n) + \sqrt{-1}\epsilon(0, H, 0) \in NS'(S)_\C. \notag 
\end{align}
Hence there is no proper sublattice $L \subsetneq NS'(S)$ such that $e^{\sqrt{-1}\epsilon H} \in L_\C$. 
Therefore the K\"ahler structure $H$ is not deformable in such a way that $e^{\omega_S} \in L_\C$ for some lattice $L \subset NS'(S)$ of rank 2. 
This calculation illustrates that $e^{B+\sqrt{-1}\omega}$ is able to detect a fine integral structure of the K\"ahler moduli space. 

\begin{Def}
A complexified K\"ahler structure $\omega_S$ is called K\"ahler rigid if there exists a rank $2$ lattice $L\subset NS'(S)$ such that $e^{\omega_S} \in L_\C$.  
\end{Def}

\begin{Thm}[cf. \cite{Kan}]
A complexified K\"ahler structure $B+\sqrt{-1}\kappa \in H^{1,1}(S)$ is K\"ahler rigid if and only if $B \in H^{1,1}(S,\Q)$ and $\kappa^2 \in H^4(S,\Q)$   
\end{Thm}
\begin{proof}
We consider an existence condition of a rank $2$ sublattice $L \subset H^*(M,\Z)$ such that 
$$
e^{B+\sqrt{-1}\kappa}=1 + B + \frac{1}{2}(B^2-\kappa^2)+\sqrt{-1}(\kappa+ B \wedge \kappa) \in L_\C. 
$$
First, $B$ needs to be rational, and hence so is $\kappa^2$. 
Then we may write $\kappa=k H$ for $k^2 \in \Q$ and $H \in H^2(S,\Z)$ with $H^2>0$. 
Indeed, in this case, there exist $m,n \in \N$ such that 
$$
m \mathrm{Re} (e^{B+\sqrt{-1}k H}), \ n \mathrm{Im} (e^{B+\sqrt{-1}kH})  \in H^*(S,\Z). 
$$
and the complexification $L_\C$ of the lattice 
$$
L=\Z m \mathrm{Re} (e^{B+\sqrt{-1} kH}) + \Z n \mathrm{Im} (e^{B+\sqrt{-1} k H}) \subset H^*(S,\Z)
$$
contains $e^{B+\sqrt{-1}k H}$. 
\end{proof}

It is natural to expect that a K\"ahler rigid K3 surface is mirror to a singular K3 surface. 
However, there is an obvious puzzle. 
The dimension of the K\"ahler moduli space of a singular K3 surface is $20$, while the dimension of the complex moduli space of a K\"ahler rigid K3 surface is at most $19$. 
It turns out that the correct framework of mirror symmetry for K3 surfaces is the generalized Calabi--Yau structures developed by Hitchin \cite{Hit} and Huybrechts \cite{Huy}. 
To solve the above puzzle, we need to incorporate deformations as a generalized K3 surface (namely a K\"ahler rigid K3 surface should be defined as a generalized K3 surface).   
A recent article \cite{Kan} investigates mirror symmetry for generalized K3 surfaces with particular emphasis on complex and K\"ahler rigid structures.  

\begin{Rem} \label{positivity by gCY}
As to Proposition \ref{positive cone}, 
the reason why the imaginary part of $\omega_S=\log \delta \in H^{1,1}(S)$ is not necessarily K\"ahler but merely positive is also well-explained from the viewpoint of the generalized Calabi--Yau structures.  
\end{Rem}


\subsection{K\"ahler constellation} \label{Kahler constellation}

From the perspective of homological mirror symmetry, the derived equivalent Calabi--Yau 3-folds share the same mirror Calabi--Yau 3-fold (remember that birationality implies derived equivalence for Calabi--Yau 3-folds \cite{Bri0}). 
It is a folklore conjecture that the complexified K\"ahler cones of the derived equivalent Calabi--Yau 3-folds give local charts of the K\"ahler moduli space $\mathfrak{M}_{\mathrm{Kah}}$ (Conjecture \ref{Bridgeland Conj}).  
Indeed, the union of the K\"ahler cones of birational Calabi--Yau 3-folds form a cone, known as the movable cone, and has been extensively studied in birational geometry. 
From this point of view, the K\"ahler constellation $\Attr_{\mathrm{Kah}}$ of a Calabi--Yau 3-fold $Y$ is defined as the union of the K\"ahler attractors of Calabi--Yau 3-folds derived equivalent to $Y$. 

In light of mirror symmetry, if $X$ and $Y$ are mirror Calabi--Yau 3-folds, then the mirror map should induce a bijective correspondence 
between the complex constellation $\Attr_{\mathrm{Cpx}}^X$ of $X$ and the K\"ahler constellation $\Attr_{\mathrm{Kah}}^Y$ of $Y$. 
$$
 \xymatrix@=18pt{
  \mathfrak{M}_{\mathrm{Cpx}}^X\ar@{}[d]|{\bigcup}  \ar@{}[r]|*{\cong} &   \mathfrak{M}_{\mathrm{Kah}}^Y \ar@{}[d]|{\bigcup} \\
 \Attr^X_{\mathrm{Cpx}}  \ar@{}[r]|*{\cong}& \Attr^Y_{\mathrm{Kah}}
}
$$
The mirror correspondence of the complex rigid and K\"ahler rigid K3 surfaces is one occurrence of such \cite{Kan} (cf. Section \ref{complex attractor mechanism for  ExS} and Section \ref{Kahler attractor for E x S}).


\appendix

\section{}

Let $p^0,q_0\in\R$ and $P=(P^{ij}),Q=(Q_{ij})\in M_3(\R)$. 
We investigate conditions on $p^0,q_0, P, Q$ under which there exist $C\in\C$ and $A=(A^{ij})\in M_3(\C)$ such that the following system of equations hold. 
\begin{align}
    \re(C) &= p^0, \label{eqn:A1} \\
    \re(CA^{ij}) &=P^{ij}, \label{eqn:A2} \\
    \re(C\Cof(A)_{ij}) &=-Q_{ij}, \label{eqn:A3} \\
    \re(C\det(A)) & =q_0. \label{eqn:A4}
\end{align}
Here $\Cof(A)_{ij}$ denotes the $(i,j)$-th entry of the cofactor matrix of $A$.

We first define 
\begin{align}
R \coloneqq & \Cof(P)+p^0Q, \notag \\
M \coloneqq & 2\det(P)+(p^0)^2q_0+p^0\tr(P^TQ), \notag \\
D \coloneqq & 2\left((\tr P^TQ)^2-\tr((P^TQ)^2)\right)-(p^0q_0+\tr(P^TQ))^2 \notag \\
 &  \ +4(p^0\det(Q)-q_0\det(P)). \notag 
\end{align}

\begin{Lem}
\label{lem:RMD}
We have the identity
\begin{equation}
\label{eqn:RMD}
4\det(R)-M^2=(p^0)^2D.
\end{equation}
\end{Lem}

\begin{proof}
We first assume that $P$ is invertible. 
Then $\det(R)$ can be expressed as 
\begin{align*}
    \det\left(\Cof(P)+p^0Q\right) & =\det(P)^{-1}\det\left(P^T\Cof(P)+p^0P^TQ\right)\\
    &=\det(P)^2\det\left(E_3+\frac{p^0}{\det(P)}P^TQ\right)
\end{align*}
where $E_3 \in M_3(\C)$ denotes the identity matrix. 
Using the standard formula 
\[
\det(E_3+B)=1+\tr(B)+\frac12\left((\tr B)^2-\tr(B^2)\right)+\det(B)
\]
for $B \in M_3(\C)$, we can check by direct computation that the Equation (\ref{eqn:RMD}) holds.

Since the Equation (\ref{eqn:RMD}) is an algebraic identity for $(p^0,q_0,P,Q) \in \R^2\times M_3(\R)^2$ 
which holds in the open dense subset $\{\det(P)\neq0\} \subset \R^2\times M_3(\R)^2$, 
it remain valid in $\R^2\times M_3(\R)^2$. 
\end{proof}

\begin{Thm}
\label{Thm:AttractorSol}
Suppose that $p^0,q_0,P,Q$ are not all zero. The following two statements are equivalent:
    \begin{enumerate}
        \item There exists $(C,A)$ satisfying (\ref{eqn:A1})-(\ref{eqn:A4}) in which $\im(A)$ is invertible;
        \item $\det(R)>0$ and $D>0$.
    \end{enumerate}
In this case, there are exactly two solutions $(C,A)$, given by
\begin{align}
C &= p^0\pm \sqrt{-1} \frac{M}{\sqrt{D}}, \notag \\
A &= \left(2PQ^T-(p^0q_0+\tr(P^TQ))E_3\mp \sqrt{-D}\right)(2R)^{-1,T}. \notag 
\end{align}

Moreover, the following two statements are equivalent:
    \begin{enumerate}
        \item There exists $(C,A)$ satisfying (\ref{eqn:A1})-(\ref{eqn:A4}) in which $A$ is symmetric and $\im(A)$ is positive definite;
        \item $P$ and $Q$ are symmetric, $R$ is positive definite, and $D>0$.
    \end{enumerate}
In this case, the unique solution is given by
\begin{align}
C &= p^0- \sqrt{-1} \frac{M}{\sqrt{D}}, \notag \\
A &= \left(2PQ-(p^0q_0+\tr(PQ))E_3+\sqrt{-D}\right)(2R)^{-1}. \notag 
\end{align}
\end{Thm}

\begin{proof}
We first prove the theorem under the assumption that $p^0\neq0$ and $M\neq0$.  
Suppose that $(C,A)$ satisfies the Equation (\ref{eqn:A1})-(\ref{eqn:A4}) and $\im(A)$ is invertible. 
By the Equation (\ref{eqn:A1}), $C \in \C$ can be written as
\[
C=p^0+i\zeta^0
\]
for some $\zeta^0\in\R$.
We denote
\[
X\coloneqq\im(A)\in M_3(\R).
\]
The Equation (\ref{eqn:A2}) gives
\begin{align*}
    P^{ij} & =\re\left((p^0+i\zeta^0)(\re(A^{ij}) +  \sqrt{-1} \im(A^{ij}))\right)\\
    & =p^0\re(A^{ij})-\zeta^0\im(A^{ij}).
\end{align*}
Since we assume $p^0\neq0$, therefore
\[
\re(A^{ij})=\frac{1}{p^0}(P^{ij}+\zeta^0 X^{ij}).
\]
Using these expressions of the real and imaginary parts of $A$, we have
\begin{align*}
    \re(CA^{ij}A^{kl}) &=\re\left((p^0+i\zeta^0)(\re(A^{ij}) +  \sqrt{-1} \im(A^{ij}))(\re(A^{kl}) +  \sqrt{-1} \im(A^{kl}))\right)\\
    &=\frac{1}{p^0}(P^{ij}P^{kl}-|C|^2X^{ij}X^{kl})
\end{align*}
for any $1\leq i,j,k,l\leq3$.
Together with the Equation (\ref{eqn:A3}), we have
\[
-Q=\re(C\Cof(A))=\frac{1}{p^0}(\Cof(P)-|C|^2\Cof(X)).
\]
Hence
\begin{equation}\label{eqn:cofactorquad}
    \Cof(X)=\frac{1}{|C|^2}(\Cof(P)+p^0Q)=\frac{1}{|C|^2}R.
\end{equation}
Since $X=\im(A)$ is invertible, we have 
$$
\det(\Cof(X))=\det(X)^2>0
$$
therefore $\det(R)>0$.

Using again the expressions of $\re(A)$ and $\im(A)$, we have
\begin{multline*}
    \re(CA^{ij}A^{kl}A^{mn})=\frac{1}{(p^0)^2}\Big(-2\zeta^0|C|^2X^{ij}X^{kl}X^{mn}+P^{ij}P^{kl}P^{mn}\\-|C|^2(P^{ij}X^{kl}X^{mn}+P^{kl}X^{ij}X^{mn}+P^{mn}X^{ij}X^{kl})\Big)
\end{multline*}
Together with the Equation (\ref{eqn:A4}) and (\ref{eqn:cofactorquad}), we obtain
\begin{align*}
    q_0&=\re(C\det(A) \\
    &=\frac{1}{(p^0)^2}\left(-2\zeta^0|C|^2\det(X)+\det(P)-|C|^2\tr(P^T\Cof(X))\right)\\
    &=\frac{1}{(p^0)^2}\left(-2\zeta^0|C|^2\det(X)+\det(P)-\tr\left(P^T(\Cof(P)+p^0Q)\right)\right)\\
    &=\frac{1}{(p^0)^2}\left(-2\zeta^0|C|^2\det(X)-2\det(P)-p^0\tr(P^TQ)\right).
\end{align*}
Hence
\begin{equation}
\label{eqn:fromA4}
2\zeta^0|C|^2\det(X)=-M
\end{equation}
by the definition of $M$. Since we assume that $M\neq0$, therefore $\zeta^0\neq0$. Using again Equation (\ref{eqn:cofactorquad}), we have
\[
\frac{M^2}{4(\zeta^0)^2|C|^4}=\det(X)^2=\det(\Cof(X))=\frac{1}{|C|^6}\det(R).
\]
By Lemma~\ref{lem:RMD},
\[
(p^0)^2D=4\det(R)-M^2=M^2\left(\frac{|C|^2}{(\zeta^0)^2}-1\right)=\frac{M^2(p^0)^2}{(\zeta^0)^2}.
\]
Hence $D>0$ and
\[
\zeta^0=\pm\frac{M}{\sqrt{D}}.
\]
By Equation (\ref{eqn:cofactorquad}) and (\ref{eqn:fromA4}), we have
\[
X=|C|^2\det(X)R^{-1,T}=\mp\sqrt{D}(2R)^{-1,T}.
\]
Hence
\begin{align*}
\re(A)&=\frac{1}{p^0}(P+\zeta^0X) \\
&=\frac{1}{p^0}(P-\frac{M}{2}R^{-1,T})=\frac{1}{p^0}(PR^T-\frac{M}{2}E_3)R^{-1,T}\\
&=\left(PQ^T-\frac{1}{2}\left(p^0q_0+\tr\left(P^TQ\right)\right)E_3\right)R^{-1,T}.
\end{align*}
Therefore
\[
A=\left(2PQ^T-(p^0q_0+\tr(P^TQ))E_3\mp \sqrt{-D}\right)(2R)^{-1,T}.
\]
This proves the statement under the assumption that $p^0\neq0$ and $M\neq0$. 

Next we assume that $p^0\neq0$ and $M=0$. 
Then the above argument up to the Equation \ref{eqn:fromA4} is still valid. 
By Lemma~\ref{lem:RMD} and (\ref{eqn:cofactorquad}), we have
\[
D=\frac{4\det(R)}{(p^0)^2}>0.
\]

Now assume that $p^0=0$. 
First we observe that the solutions are still valid in this case. 
Hence it suffices to show that they are the only two solutions.
Since $p^0,q_0,P,Q$ are not all zero, we have $c=i\zeta^0\neq0$. Then Equation (\ref{eqn:A2}) gives
\[
X\coloneqq\im(A)=\frac{-1}{\zeta^0}P.
\]
Note that $P$ is invertible since $X=\im(A)$ is assumed to be invertible.
We denote
\[
Y\coloneqq\re(A).
\]
Then
\begin{align*}
    \re(CA^{ij}A^{kl}) &=\re\left((i\zeta^0)(\re(A^{ij}) + \sqrt{-1} \im(A^{ij}))(\re(A^{kl}) +  \sqrt{-1} \im(A^{kl}))\right)\\
    &=P^{ij}Y^{kl}+P^{kl}Y^{ij}.
\end{align*}
By Equation (\ref{eqn:A3}), we have
\[
\begin{pmatrix}
0 & 0 & 0 & 0 & P^{33} & -P^{32} & 0 & -P^{23} & P^{22} \\
0 & 0 & 0 & -P^{33} & 0 & P^{31} & P^{23} & 0 & -P^{21} \\
0 & 0 & 0 & P^{32} & -P^{31} & 0 & -P^{22} & P^{21} & 0 \\
0 & -P^{33} & P^{32} & 0 & 0 & 0 & 0 & P^{13} & -P^{12} \\
P^{33} & 0 & -P^{31} & 0 & 0 & 0 & -P^{13} & 0 & P^{11} \\
-P^{32} & P^{31} & 0 & 0 & 0 & 0 & P^{12} & -P^{11} & 0 \\
0 & P^{23} & -P^{22} & 0 & -P^{13} & P^{12} & 0 & 0 & 0 \\
-P^{23} & 0 & P^{21} & P^{13} & 0 & -P^{11} & 0 & 0 & 0 \\
P^{22} & -P^{21} & 0 & -P^{12} & P^{11} & 0 & 0 & 0 & 0 \\
\end{pmatrix}
\begin{pmatrix}
Y^{11}\\Y^{12}\\Y^{13}\\
Y^{21}\\Y^{22}\\Y^{23}\\
Y^{31}\\Y^{32}\\Y^{33}\\
\end{pmatrix}
=-
\begin{pmatrix}
Q_{11}\\Q_{12}\\Q_{13}\\
Q_{21}\\Q_{22}\\Q_{23}\\
Q_{31}\\Q_{32}\\Q_{33}\\
\end{pmatrix}
\]
The $9\times9$ matrix on the left hand side has determinant $-2(\det(P))^3\neq0$, hence $Y$ can solved uniquely.
\[
\re(CA^{ij}A^{kl}A^{mn})=(P^{ij}Y^{kl}Y^{mn}+P^{kl}Y^{ij}Y^{mn}+P^{mn}Y^{ij}Y^{kl})-\frac{1}{(\zeta^0)^2}P^{ij}P^{kl}P^{mn}.
\]
By Equation (\ref{eqn:A4}),
\[
q_0=\tr(P^T\Cof(Y))-\frac{1}{(\zeta^0)^2}\det(P).
\]
This solves $\zeta^0$ up to sign. This proves that there are no other solutions.
\end{proof}

\section{}

\begin{Thm}[Thoerem \ref{attractor = 9}]
The complex constellation $\Attr_{\mathrm{Cpx}}$ bijectively corresponds to the abelian 3-folds with Picard number $9$.  
\end{Thm}
\begin{proof}
It suffices to show that an abelian 3-fold $X_T = \C^3/(\Z^3 + T \Z^3)$ with $\rho(X_T)=9$ is a complex attractor variety. 
By Theorem \ref{maximal picard number}, $X_T$ is isogenous to $(E_{\sqrt{-D}})^3$ for some $D\in\N$. 
Therefore there exists a lattice embedding $\Z^3 + \sqrt{-D}\Z^3 \hookrightarrow \Z^3 + T \Z^3$ induced by $R\in M_3(\Z)$ such that
\[
TR\in M_3(\Z)+\sqrt{-D}E_3.
\]
Then $R=\sqrt{D} \mathrm{Im}(T)^{-1}$ is symmetric and positive definite. 
We will find a $3$-cycle $\gamma \in H_3(X_T,\Z)$ for which $T$ is a complex attractor point.  
We write $\gamma$ as 
$$
\gamma \ = q_0 A_0+ \sum_{i,j=1}^3Q_{ij} A_{ij}+ \sum_{i,j=1}^3P^{ij} B^{ij}+p^0 B^0
$$
as in Section \ref{complex attractor mechanism for torus}. 
We introduce 
\begin{itemize}
\item $n\coloneqq(D+1)\det(R)\in\N$.
\item $M\coloneqq2n\det(R)\in\N$.
\item $S\coloneqq 2n(TR-\sqrt{-D}E_3)=2n\mathrm{Re}(T)R\in M_3(\Z)$.
\item $p^0\coloneqq\det(R)\in\Z$.
\item $P\coloneqq(p^0S+ME_3)(2nR)^{-1}\in M_3(\Q)$.
\item $Q\coloneqq\frac{1}{p^0}(nR-\mathrm{Cof}(P))\in M_3(Q)$.
\item $q_0\coloneqq\frac{1}{(p^0)^2}(2n\det(R)-2\det(P)-p^0\tr(PQ))\in\Q$.
\end{itemize}
Now we show that $T$ is an attractor of $\gamma$ given by $(p^0,P,Q,q_0)$. 
Define as in Appendix A:
\begin{align}
\widetilde R  \coloneqq & \Cof(P)+p^0Q, \notag \\
\widetilde M  \coloneqq & 2\det(P)+(p^0)^2q_0+p^0\tr(PQ),\notag \\
\widetilde D  \coloneqq & 2\left((\tr PQ)^2-\tr((PQ)^2)\right)-(p^0q_0+\tr(PQ))^2 \notag \\
& \ +4(p^0\det(Q)-q_0\det(P)). \notag 
\end{align}
Then we have 
$$
\widetilde R=nR, \ \ \ \widetilde M=2n\det(R)=M.
$$ 
Moreover, by Lemma~\ref{lem:RMD}, 
\begin{align*}
	\widetilde D & =\frac{1}{(p^0)^2}(4\det(\widetilde R)-\widetilde M^2) \\
	& = \frac{1}{\det(R)^2}(4n^3\det(R)-4n^2\det(R)^2)\\
	& = 4(D+1)^2\det(R)(n-\det(R))\\
	& = 4D(D+1)^2\det(R)^2\\
	& = 4n^2D>0.
\end{align*}
By Theorem~\ref{Thm:AttractorSol}, an attractor is given by
\[
\widetilde T=\left(2PQ-(p^0q_0+\tr(PQ))E_3+\sqrt{-\widetilde D}\right)(2\widetilde R)^{-1}.
\]
Hence
\[
\mathrm{Im}(\widetilde T)=\sqrt{\widetilde D}(2\widetilde R)^{-1}=\sqrt{D}R^{-1}=\mathrm{Im}(T),
\]
and
\begin{align*}
2n\mathrm{Re}(\widetilde T)R & =2PQ-(p^0q_0+\tr(PQ))E_3\\
& = \frac{2}{p^0}(P\widetilde R-\det(P)E_3) - \frac{1}{p^0}(\widetilde M-2\det(P))E_3\\
& = \frac{1}{p^0}(2P\widetilde R-\widetilde M E_3)\\
& = \frac{1}{p^0}(2(p^0S+ME_3)(2nR)^{-1}(nR)-ME_3)\\
& = S = 2n\mathrm{Re}(T)R
\end{align*}
since
$P\widetilde R=\det(P)E_3+p^0PQ$.
Therefore we have $\widetilde T=T$, so $T$ is an attractor. 
\end{proof}


\par\noindent{\scshape \small
Department of Mathematics, Evans Hall \\
University of California, Berkeley \\
Berkeley, CA 94720}
\par\noindent{\ttfamily ywfan@berkeley.edu}
\ \\
\par\noindent{\scshape \small
Faculty of Policy Management, Keio University \\
Endo 5322, Fujisawa, Kanagawa, 252-0882, Japan}
\par\noindent{\ttfamily  atsushik@sfc.keio.ac.jp}
\end{document}